\documentclass{scrartcl}

\usepackage{scrhack}
\usepackage[utf8]{inputenc}
\usepackage{mathtools}
\usepackage{mathrsfs}
\usepackage{amsmath}
\usepackage{amsfonts}
\usepackage{amssymb}
\usepackage{amsthm}
\usepackage{enumitem}
\usepackage{ifthen}
\usepackage{xspace}
\setenumerate{label=\textup{(\alph*)}}

\usepackage[hyphens]{url}
\usepackage[hidelinks]{hyperref}
\usepackage{cleveref}
\makeatletter
\DeclareOldFontCommand{\sc}{\normalfont\scshape}{\@nomath\sc}
\makeatother

\crefdefaultlabelformat{#2\textup{#1}#3}
\newtheorem{definition}{Definition}
\numberwithin{definition}{section}
\newtheorem{Theorem}[definition]{Theorem}
\newtheorem{lemma}[definition]{Lemma}

\newtheorem{assumption}[definition]{Assumption}

\theoremstyle{definition}

\crefname{customlem}{Lemma}{Lemmas}
\Crefname{customlem}{Lemma}{Lemmas}
\crefname{customthm}{Theorem}{Theorems}
\Crefname{customthm}{Theorem}{Theorems}

\crefname{corollary}{Corollary}{Corollaries}
\Crefname{assumption}{Assumption}{Assumptions}
\crefname{theorem}{Theorem}{Theorems}
\crefname{lemma}{Lemma}{Lemmas}
\crefname{proposition}{Proposition}{Propositions}

\numberwithin{equation}{section}
\crefname{enumi}{Assumption}{Assumptions}

\newlist{enumthm}{enumerate}{1}
\setlist[enumthm]{label=\textup{(\alph*)},ref=\theassumption~\textup{(\alph*)}}
\crefalias{enumthmi}{assumption}

\usepackage{subfig}

\newcommand{\functionArgument}[1]{\ifthenelse{\equal{#1}{}}  
	{}
	{({#1})}
}

\newcommand{\coperator}{K}
\newcommand{\maxo}[1]{(#1)_+}
\newcommand{\dualp}[3][]{\langle #2, #3 \rangle_{{#1}^*\!, #1}}

\newcommand{\dualpstar}[4][]{\langle #3, #4 \rangle_{#1, #2}}
\newcommand{\dualpHzeroone}[3][]
{\langle #2, #3 \rangle_{H^{-1}(\domain), H_0^1(\domain)}}

\newcommand{\wprox}[3][]{\mathrm{prox}_{#2}\functionArgument{#3}}
\newcommand{\prox}[2]{\wprox[]{#1}{#2}}
\newcommand{\vadcsp}{V_{\text{ad}}}
\newcommand{\norm}[2][2]{\left\lVert#2\right\rVert_{#1}}
\renewcommand{\norm}[2][2]{\|#2\|_{#1}}

\DeclarePairedDelimiterXPP\cnorm[2]{}\lVert\rVert{_{#1}}{#2}
\newcommand{\radius}{R_{\text{ad}}}
\newcommand{\rdc}{\kappa}
\newcommand{\rrhs}{b}

\newcommand{\inner}[3][]{( #2, #3 )_{#1}}
\newcommand{\obj}{f}
\newcommand{\pobj}{J}
\newcommand{\rpobj}{\widehat{\pobj}}

\newcommand{\erpobj}{F}
\newcommand{\dist}[2]{\mathrm{dist}\functionArgument{#1,#2}}
\newcommand{\deviation}[2]{\mathbb{D}\functionArgument{#1,#2}}
\newcommand{\csp}{U}
\newcommand{\adcsp}{\csp_\text{ad}}
\newcommand{\wadcsp}{W_\text{ad}}
\newcommand{\ssp}{Y}

\newcommand{\asp}{Z}

\newcommand{\domain}{D}

\DeclareMathOperator*{\argmin}{arg\,min}
\renewcommand{\natural}{\mathbb{N}}
\newcommand{\real}{\mathbb{R}}
\newcommand{\embedding}{\xhookrightarrow{}}

\newcommand{\bsp}{V}
\newcommand{\eu}{\ensuremath{\mathrm{e}}}
\newcommand{\du}{\ensuremath{\mathrm{d}}}
\newcommand{\Du}{\ensuremath{\mathrm{D}}}
\newcommand{\wpone}{w.p.~$1$\xspace}

\newcommand{\tas}{\text{as}}
\newcommand{\as}{\tas}
\newcommand{\frechet}{Fr\'echet}

\newcommand{\Caratheodory}{Carath\'eodory}
\newcommand{\hoelder}{H\"older}
\newcommand{\friedrichs}{Friedrichs\xspace}

\DeclarePairedDelimiterXPP\cE[1]{\mathbb{E}}[]{}{%
	
	#1}

\newcommand{\cA}{\mathcal{A}}
\newcommand{\cF}{\mathcal{F}}

\renewcommand{\abstract}[1]
{
	{\small
		\noindent\textbf{Abstract.} {#1}
		\\
	}
}

\newcommand{\keywords}[1]
{
	{\small
		\noindent\textbf{Key words.} {#1}
		\\
	}
}

\newcommand{\subclass}[1]
{
	{\small
		\noindent\textbf{AMS subject classifications.} {#1}
	}
}

\title{Consistency of Monte Carlo Estimators for Risk-Neutral
	PDE-Constrained Optimization}

\author{Johannes Milz\thanks{Technical University of Munich,
		Department of Mathematics, Boltzmannstr.\ 3, 85748 Garching, Germany,
		\texttt{milz@ma.tum.de}}}

\date{November 16, 2022}

\begin{document}

\maketitle


\abstract{%
	We apply the sample average approximation (SAA)
	method
	to risk-neutral optimization problems
	governed by nonlinear partial differential equations (PDEs)
	with random inputs.
	We analyze the consistency of the
	SAA optimal values and
	SAA solutions.
	Our analysis exploits problem structure in
	PDE-constrained optimization problems, allowing us to 
	construct deterministic, compact subsets 
	of the feasible set that contain the solutions to
	the risk-neutral problem and eventually those to the SAA problems.
	The construction is used to
	study the consistency using results established in the literature
	on stochastic programming.
	The assumptions of our framework are verified on
	three nonlinear optimization problems under uncertainty.
}


\subclass{65C05, 90C15, 35R60, 90C48, 90C30, 60H25}

\section{Introduction}
\label{sec:intro}

Advances in areas such as computational science
and engineering, applied mathematics, software design,
and scientific computing have allowed decision makers to optimize complex 
physics-based systems under uncertainty, such as those 
modeled using partial differential equations (PDEs) with uncertain inputs. 
Recent applications in the field of PDE-constrained optimization 
under uncertainty are, for example,
oil field development \cite{Nasir2021},
stellarator coil optimization
\cite{Wechsung2021}, acoustic wave propagation \cite{Yang2017}, 
and shape optimization of electrical engines
\cite{Kolvenbach2018}.
In the literature on optimization under uncertainty, several approaches
have been proposed for obtaining decisions that are resilient to
uncertainty, such as robust optimization \cite{Ben-Tal2009}
and stochastic optimization \cite{Shapiro2003}.
When the parameter vector is modeled as a random vector with 
known probability distribution, a common approach is to 
seek decisions that minimize the expected value of a parameterized 
objective function. 
The resulting optimization problem is referred to as
risk-neutral optimization problem.
However, evaluating the risk-neutral problem's objective function
would require computing a potentially high-dimensional integral. 
Furthermore, each evaluation of 
the parameterized objective function may require the simulation of 
complex systems of PDEs, adding another challenge to obtaining
solutions to risk-neutral PDE-constrained optimization problems.

A common approach for approximating risk-neutral optimization problems 
is the sample average approximation (SAA) method, yielding the SAA problem.
For example,
the SAA approach is used in the literature on mathematical 
programming
\cite{Kleywegt2002,Shapiro2003,Shapiro2005} and on
PDE-constrained optimization 
\cite{Haber2012,Kouri2016,Roemisch2021,Wechsung2021}.
The SAA problem's objective function is the sample average of the
parameterized objective function computed using samples of the random vector.
To assess 
the quality of the SAA solutions as approximate solutions to the
risk-neutral problem, different error measures have been considered, such as the
consistency of the SAA optimal value and
of SAA solutions
\cite{Shapiro1993,Shapiro1991,Shapiro2014,Artstein1995,Hess1996,Lachout2005}, 
nonasymptotic sample size estimates 
\cite{Shapiro2003,Shapiro2005,Shapiro2008,Cucker2002,Royset2019},
mean and almost sure convergence rates \cite{Banholzer2019},
and confidence intervals for  SAA optimal values \cite{Guigues2017}.

A number of results on the SAA approach are based on the compactness
of either the feasible set or of sets that eventually 
contain the SAA solutions, such as the consistency properties of SAA solutions
and sample size estimates.
The analysis of the SAA approach as applied to PDE-constrained 
optimization problems is complicated by the fact that the feasible sets
are commonly noncompact, such as the set of square integrable
functions defined on the interval $(0,1)$ with values in $[-1,1]$. 
Moreover, level sets of the SAA objective function 
may not be contained in a deterministic, compact set
as shown in \cref{sec:saalevelsets}. 
Our approach for establishing consistency is based on 
that developed in \cite[Chap.\ 5]{Shapiro2014}. 
While the consistency results  in \cite[Chap.\ 5]{Shapiro2014} are established
for  finite dimensional stochastic programs, the results do not require the
compactness of the feasible set. Instead, they are valid provided that the 
solution set to the stochastic program and those to the SAA problems
are eventually contained in a deterministic, compact set. 

We establish the consistency of SAA optimal values
and SAA solutions to 
risk-neutral nonlinear PDE-constrained optimization problems. 
For analyzing the SAA approach,
we construct  deterministic, compact subsets
of the possibly infinite dimensional feasible sets
that contain the solutions
to risk-neutral PDE-constrained problems and
eventually those to the corresponding SAA problems.
This observation allows us to study the consistency using 
the tools developed in the literature on M-estimation
\cite{Huber1967,LeCam1953}
and  stochastic programming 
\cite{Shapiro2003,Shapiro2014}.
Our consistency results are inspired by and based on those established
in \cite[Sects.\ 2 and 7]{Shapiro2003} and \cite[Chap.\ 5]{Shapiro2014}.
For our construction of these compact sets, we use the fact that 
many PDE-constrained optimization problems involve compact operators, 
such as compact embeddings. Moreover, we use 
first-order optimality conditions and
PDE stability estimates.
The construction  is partly inspired by the computations used to 
establish higher regularity of solutions to deterministic
PDE-constrained optimization problems 
\cite[p.\ 1305]{Meidner2008},
\cite[Sect.\ 2.15]{Troeltzsch2010a}
and a computation made in the  author's dissertation 
\cite[Sect.\ 3.5]{Milz2021a} which demonstrates that all
SAA solutions to certain linear elliptic optimal control problems
are contained in a compact set.

The SAA method as applied to risk-neutral strongly convex
PDE-constrained optimization has recently been analyzed in
\cite{Hoffhues2020,Martin2021,Roemisch2021,Milz2021}.
The authors of \cite{Phelps2016} apply the SAA scheme to the optimal 
control of ordinary differential equations with random inputs and demonstrate
the epiconvergence of the  SAA objective function and
the weak consistency of SAA critical points in
the sense defined in \cite[Def.\ 3.3.6]{Polak1997}.
The weak consistency implies that accumulation points
of SAA critical points are critical points of the optimal control problem
\cite[p.\ 13]{Phelps2016}.

Monte Carlo sampling is one approach to approximating
expected values in stochastic program's objective functions.
For  strongly convex elliptic PDE-constrained optimization problems, 
quasi-Monte Carlo techniques are
analyzed in \cite{Guth2019}. 
Further discretization approaches for expectations are, for example,
stochastic collocation \cite{Tiesler2012} and low-rank tensor approximations
\cite{Garreis2017}. 
Besides risk-neutral PDE-constrained optimization,
risk-averse PDE-constrained optimization 
\cite{Alexanderian2017,Conti2011,Kouri2016,Kouri2018a,MartnezFrutos2018}, 
distributionally robust PDE-constrained optimization
\cite{Kouri2017,Milz2020a}, 
robust PDE-constrained optimization
\cite{Alla2019,Kolvenbach2018,Lass2017}, and
PDE-constrained optimization with chance constraints
\cite{Chen2020,Farshbaf-Shaker2020,Farshbaf-Shaker2018,Geletu2020,Tong2022}
provide approaches to decision making under uncertainty
with PDEs.

\subsection*{Outline}
We introduce notation in \cref{sec:notation}
and a class of risk-neutral nonlinear PDE-constrained optimization problems
and their SAA problems in \cref{sec:assumptions}.
\Cref{sec:compact_subsets} presents a compact 
subset that contains the solutions to the risk-neutral
problem and eventually those to its SAA problems.
We study the consistency of SAA optimal values and solutions 
in \cref{subsect:consistency_solutions}.
\Cref{sec:applications} discusses the application of our theory
to three nonlinear PDE-constrained optimization problems under uncertainty. 
We summarize our contributions, and discuss
some limitations of our approach and open research questions in 
\cref{sec:conclusion}.

\section{Notation and Preliminaries}
\label{sec:notation}
Throughout the paper, the control space $\csp$ is a 
real, separable Hilbert space
and is  identified with its dual, that is, we
omit writing the Riesz mapping.

Metric spaces are defined over the real numbers
and equipped with their Borel sigma-algebra. 
We abbreviate ``with probability one'' by \wpone.
Let $(\Theta, \cA, \mu)$ be probability space.
For two complete metric spaces 
$\Lambda_1$  and $\Lambda_2$, a mapping 
$G : \Lambda_1 \times \Theta \to \Lambda_2$
is a \Caratheodory\ mapping
if $G(\cdot, \theta)$ is continuous for all $\theta \in \Theta$ and
$G(v, \cdot)$ is measurable for each $v \in \Lambda_1$.
Let $\Lambda$ be a Banach space.
For each $N \in \natural$,
let $\Upsilon_N : \Theta  \rightrightarrows \Lambda$ be a 
set-valued mapping and let $\Psi \subset \Lambda$ be a set.
We say that \wpone for all sufficiently large
$N$, $\Upsilon_N \subset  \Psi$ if the set
$\{\, \theta \in \Theta \colon \, \exists \, n(\theta) \in \natural\, 
\forall \, N \geq n(\theta); \, \Upsilon_N(\theta) \subset \Psi \,\}$
is contained in $\cA$ and occurs \wpone,
that is, if the limit inferior of the sequence
$
(\{\, \theta \in \Theta \colon \,\Upsilon_N(\theta) \subset \Psi \,\})_N
$
is contained in $\cA$ and occurs \wpone
\cite[p.\ 55]{Billingsley2012}.
A mapping $\upsilon : \Theta \to \Lambda$ is strongly measurable
if there exists a sequence of simple mappings
$\upsilon_k : \Theta \to \Lambda$ such that
$\upsilon_k(\theta) \to  \upsilon(\theta)$ as $k \to \infty$
for all $\theta \in \Theta$
\cite[Def.\ 1.1.4]{Hytoenen2016}.
If $\Lambda$ is separable, then
$\upsilon : \Theta \to \Lambda$ is strongly measurable
if and only if it is measurable
\cite[Cor.\ 1.1.2 and Thm.\ 1.1.6]{Hytoenen2016}.
The dual to a Banach space $\Lambda$ is 
$\Lambda^*$ and the norm of $\Lambda$ is denoted 
by $\norm[\Lambda]{\cdot}$. We use $\dualp[\Lambda]{\cdot}{\cdot}$ 
to denote 
the dual pairing between $\Lambda^*$ and $\Lambda$.
If $\Lambda$ is a reflexive Banach space, 
we identify $(\Lambda^*)^*$ with $\Lambda$ and write
$(\Lambda^*)^* = \Lambda$.
Let $\Lambda_1$ and $\Lambda_2$ be real Banach spaces.
A linear operator $\Upsilon : \Lambda_1 \to \Lambda_2$ is compact if
the image $\Upsilon(\Lambda_0)$ is 
precompact in $\Lambda_2$
for each bounded set $\Lambda_0 \subset \Lambda_1$
\cite[Def.\ 8.1-1]{Kreyszig1978}.
The operator
$\Upsilon^* \colon \Lambda_2^* \to \Lambda_1^*$  is the 
(Banach space-)adjoint operator
of the linear, bounded mapping $\Upsilon \colon \Lambda_1 \to \Lambda_2$ and
is defined by
$\langle \Upsilon^*v_2, v_1\rangle_{\Lambda_1^*,\Lambda_1}
= \langle  v_2, \Upsilon v_1\rangle_{\Lambda_2^*,\Lambda_2}$
\cite[Def.\ 4.5-1]{Kreyszig1978}.
We use
$\Lambda_1 \embedding \Lambda_2$  to denote a continuous embedding
from $\Lambda_1$ to $\Lambda_2$, that is, $\Lambda_1 \subset \Lambda_2$ and the 
embedding 
operator $\iota : \Lambda_1 \to \Lambda_2$ defined by $\iota[v] = v$
is continuous  \cite[Def.\ 7.15 and Rem.\ 7.17]{Renardy2004}.
A continuous embedding is compact if the embedding operator
is a compact operator \cite[Def.\ 7.25 and Lem.\ 8.75]{Renardy2004}.
We denote by $\Du \obj$ the \frechet\ derivative
of $\obj$, and use the notation $\Du_x \obj$ and $\obj_x$
for partial derivatives with respect to $x$. 
Throughout the text,  $\domain \subset \real^d$ is a bounded domain.
For $p \in [1,\infty)$, 
we denote by $L^p(\domain)$ the Lebesgue space of $p$-integrable
functions defined on $\domain$ and
$L^\infty(\domain)$ that of essentially bounded functions.
The space $H^1(\domain)$ is the space of all $v \in L^2(\domain)$
with weak derivatives contained in $L^2(\domain)^d$,
where $L^2(\domain)^d$ is the Cartesian product of $L^2(\domain)$ taken
$d$ times.
We equip $H^1(\domain)$ with the norm
$\norm[H^1(\domain)]{y} = 
(\norm[L^2(\domain)]{y}^2 + \norm[L^2(\domain)^d]{\nabla y}^2)^{1/2}$.
The Hilbert space $H_0^1(\domain)$
consists of all $v \in H^1(\domain)$
with zero boundary traces and is equipped with the norm
$\norm[H_0^1(\domain)]{y} = \norm[L^2(\domain)^d]{\nabla y}$.
We define $H^{-1}(\domain) = H_0^1(\domain)^*$.
We define \friedrichs' constant $C_\domain \in (0,\infty)$ by
$
C_{D} = \sup_{v \in H_0^1(\domain) \setminus \{0\}}\,
\norm[L^2(\domain)]{v}/\norm[H_0^1(\domain)]{v}
$.
The indicator function $I_{\csp_0} : \csp \to [0,\infty]$ of 
$\csp_0\subset \csp$ is given
by $I_{\csp_0}(v) = 0$ if $v \in \csp_0$ and $I_{\csp_0}(v) = \infty$
otherwise. For a convex, lower semicontinuous, 
proper function $\chi : \csp \to (-\infty,\infty]$, 
the proximity operator $\prox{\chi}{}:\csp \to \csp$ of $\chi$ is defined by
(see \cite[Def.\ 12.23]{Bauschke2011})
\begin{align*}
	\prox{\chi}{v}
	= \argmin_{w\in\csp}\, 
	\chi(w) + (1/2)\norm[\csp]{v-w}^2.
\end{align*}

\section{Risk-neutral PDE-constrained optimization problem}
\label{sec:assumptions}

We consider the risk-neutral PDE-constrained optimization problem
\begin{align}
	\label{eq:ocp}
	\min_{u\in\csp}\, 
	\cE{\pobj_1(S(u,\xi),\xi)} + \psi(u) +
	(\alpha/2)\norm[\csp]{u}^2\,
\end{align}
and its sample average approximation 
\begin{align}
	\label{eq:saa}
	\min_{u\in\csp}\, 
	\frac{1}{N} \sum_{i=1}^N\pobj_1(S(u,\xi^i),\xi^i)
	+ \psi(u) +
	(\alpha/2)\norm[\csp]{u}^2,
\end{align}
where $\alpha > 0$, and $\xi^1$, $\xi^2, \ldots$ are independent 
identically distributed $\Xi$-valued
random elements defined on a complete
probability space $(\Omega, \cF, P)$
and each $\xi^i$ has the same distribution 
as that of the random element $\xi$. Here, $\xi$ maps
from a probability space to a complete probability space
with sample space $\Xi$ being a complete, separable metric space.
Moreover, $\pobj_1 \colon \ssp \times \Xi \to [0,\infty)$, 
$\psi \colon \csp \to [0,\infty]$, and
$S : \csp \times \Xi \to \ssp$. We state assumptions
on these mappings and on
the control space $\csp$ and state space $\ssp$ in \Cref{ass:pobj,ass:E}.
Let $\erpobj : \csp \to [0,\infty]$ be the objective function of \eqref{eq:ocp} 
and let $\hat{\erpobj}_N : \csp \to [0,\infty]$ be that of \eqref{eq:saa}. 
Since $\xi^1, \xi^2, \ldots$ are defined on the common probability space
$(\Omega,\cF,P)$, we can view the function $\hat{\erpobj}_N$ as defined
on $\csp \times \Omega$. However, we often omit writing the second argument.
We often use $\xi$ to denote a deterministic element in $\Xi$.

In the remainder of the section, we impose conditions
on the optimization problem \eqref{eq:ocp}.
\Cref{ass:pobj,ass:E} ensure that the reduced formulation of the
risk-neutral problem \eqref{eq:ocp}
and its SAA problem \eqref{eq:saa} are well-defined.

\begin{assumption}
	\label{ass:pobj}
	\begin{enumthm}[nosep,leftmargin=*]
		\item The space $\csp$ is a real, separable Hilbert space, 
		and $\ssp$ is a real, separable Banach space.
		
		\item 
		\label{ass:pobj_2}
		The function $\pobj_1 : \ssp  \times  \Xi \to [0,\infty)$ 
		is a \Caratheodory\ function, 
		and $\pobj_1(\cdot,  \xi)$ is continuously differentiable
		for all $\xi \in \Xi$.
		\item  
		\label{ass:pobj_3}
		The regularization parameter $\alpha$ is positive,
		and $\psi : \csp \to [0,\infty]$
		is proper, convex and lower semicontinuous.
	\end{enumthm}
\end{assumption}%
The nonnegativity of $\pobj_1$
and  $\psi$ is fulfilled for many
PDE-constrained optimization problems
(see \cref{sec:applications}). 
We define the feasible set 
\begin{align}
	\label{eq:adcsp}
	\adcsp = \{ \, u \in \csp :\, \psi(u)<\infty\,\}.
\end{align}%

\begin{assumption}
	\label{ass:E}
	\begin{enumthm}[nosep,leftmargin=*]
		\item The operator $E : (\ssp \times \csp) \times \Xi \to \asp$
		is a  \Caratheodory\ mapping, 
		$E(\cdot,\cdot, \xi)$ is continuously differentiable
		for all $\xi \in \Xi$, 
		and $\asp$ is a real, separable Banach space. 
		\item 	For each $(u,\xi) \in \csp \times \Xi$, 
		$S(u,\xi) \in \ssp$ is the unique solution 
		to: find $y \in \ssp$ with
		$E(y,u,\xi) = 0$.
		\item For each $(u,\xi) \in \csp \times \Xi$, 
		$E_y(S(u,\xi),u,\xi)$ has a bounded inverse.
	\end{enumthm}
\end{assumption}
\Cref{ass:pobj,ass:E} and the implicit
function theorem 
ensure that
$S(\cdot, \xi) $ is continuously differentiable on 
$\csp$ for each $\xi \in \Xi$.
Let us define 
$\rpobj_1: \csp \times  \Xi \to [0,\infty)$ by
\begin{align}
	\label{eq:rpobj1}
	\rpobj_1(u,\xi) = \pobj_1(S(u,\xi),\xi)
\end{align}
and  $\rpobj : \csp \times \Xi \to [0,\infty)$ by
\begin{align}
	\label{eq:rpobj}
	\rpobj(u,\xi) = \rpobj_1(u,\xi) + \psi(u) + (\alpha/2)\norm[\csp]{u}^2.
\end{align}
Let us fix $\xi \in \Xi$. 
\Cref{ass:pobj,ass:E}  allow us to use
the adjoint approach \cite[Sect.\ 1.6.2]{Hinze2009} 
to compute the gradient of 
the function
$\rpobj_1(\cdot,\xi)$ defined in \eqref{eq:rpobj1} at each $u \in \csp$.
It yields the gradient
\begin{align}
	\label{eq:Euz}
	\nabla_u \rpobj_1(u,\xi) = E_u(S(u,\xi),u,\xi)^*z(u,\xi),
\end{align}
where for each $(u,\xi) \in \csp \times \Xi$,
$z(u,\xi) \in \asp^*$ is the unique solution to the
(parameterized) adjoint equation: find $z \in \asp^*$
with 
\begin{align}
	\label{eq:aeq}
	E_y(S(u,\xi),u,\xi)^*z = - \Du_y \pobj_1(S(u,\xi),\xi).
\end{align}%

\begin{assumption}
	\label{ass:existence}
	The risk-neutral problem \eqref{eq:ocp} has a solution.
	For each $N \in \natural$
	and every $\omega \in \Omega$, 
	the SAA problem \eqref{eq:saa} has a solution.
\end{assumption}%
We refer the reader to 
\cite[Thm.\ 1]{Kouri2018}
and
\cite[Prop.\ 3.12]{Kouri2018a}
for theorems on the existence of solutions to 
risk-averse PDE-constrained optimization problems.

For some $u_0 \in \adcsp$  
with $\cE{\rpobj(u_0,\xi)} < \infty$ and a scalar $\rho \in  (0,\infty)$,
we define the  set
$$\vadcsp^\rho(u_0) = \{\, u \in \adcsp \colon\, (\alpha/2) \norm[\csp]{u}^2 \leq 
\cE{\rpobj(u_0,\xi)} + \rho \,\}.$$
The existence of such a point $u_0$ is implied
by \Cref{ass:existence}, for example. Whereas $\adcsp$ may be unbounded, 
the set $\vadcsp^\rho(u_0)$ is bounded. If $\adcsp$ is bounded and 
$\rho \in (0,\infty)$ is sufficiently large, then $\vadcsp^\rho(u_0) = \adcsp$.
Each solution 
to the risk-neutral problem \eqref{eq:ocp} is contained
in $\vadcsp^\rho(u_0)$, because $\rpobj_1 \geq 0$, $\psi \geq 0$, 
and $u_0 \in \adcsp$.

\Cref{ass:pobjE} allows us to construct  compact subsets
of the bounded set $\vadcsp^\rho(u_0)$. 

\begin{assumption}
\label{ass:pobjE}
\begin{enumthm}[nosep,leftmargin=*]
	\item 
	\label{ass:pobjE_operator}
	The linear operator 
	$\coperator: \bsp \to \csp$
	is compact,  $\bsp$ is a real, separable Banach space,
	and $B_{\vadcsp^\rho(u_0)} \subset \csp$ is 
	a bounded, convex
	neighborhood of $\vadcsp^\rho(u_0)$.
	\item 
	\label{ass:pobjE_gradient}
	The mapping 
	$M :\csp 
	\times \Xi \to \bsp$ is a \Caratheodory\ mapping
	and for all $(u,\xi) \in \csp \times \Xi$,
	\begin{align}
		\label{eq:lifted_rpobj_1}
		\nabla_u
		\rpobj_1(u,\xi) = \coperator[M(u,\xi)].
	\end{align}

	\item 
	\label{ass:pobjE_stability}
	For some integrable random variable
	$\zeta : \Xi \to [0,\infty)$, 
	\begin{align}
		\label{eq:bound_Euz}
		\norm[\bsp]{M(u,\xi)}
		\leq 		\zeta(\xi)
		\quad \text{for all} \quad (u,\xi)\in B_{\vadcsp^\rho(u_0)} \times \Xi. 
	\end{align}
\end{enumthm}
\end{assumption}
\Cref{ass:pobjE_gradient} and the gradient formula in \eqref{eq:Euz} yield
for all $(u,\xi) \in \csp \times \Xi$,
\begin{align*}
E_u(S(u,\xi),u,\xi)^*z(u,\xi) = \coperator[M(u,\xi)].
\end{align*}
\Cref{ass:pobjE_stability} may be verified using
stability estimates for the solution operator and adjoint state.
If $B_{\vadcsp^\rho(u_0)}$ would be unbounded, then 
\Cref{ass:pobjE_stability} may be violated.

\begin{lemma}
\label{lem:rpobj1islsc}
If \Cref{ass:pobj,ass:E} hold,
then 	$\rpobj_1 : \csp \times \Xi \to [0,\infty)$ is a
\Caratheodory\ mapping.
\end{lemma}%
\begin{proof}
For each $\xi \in \Xi$, the implicit function
theorem when combined with 
\Cref{ass:pobj,ass:E}
ensures that the mappings $S(\cdot,\xi)$ is continuously
differentiable. In particular, $\rpobj_1(\cdot,\xi)$
is continuous. 	Fix  $u \in \csp$.
The measurability of $S(u,\cdot)$
follows from \cite[Thm.\ 8.2.9]{Aubin2009}
when combined with 	\Cref{ass:pobj,ass:E}.
Using the definition of $\rpobj_1$ provided in \eqref{eq:rpobj1}, 
the measurability of $\pobj_1(u,\cdot)$ and of
$S(u,\cdot)$, the separability of $\ssp$, and the composition rule
\cite[Cor.\ 1.1.11]{Hytoenen2016}, we find that
$\rpobj_1 (u,\cdot)$ is measurable. 
\end{proof}

We define the expectation function
$\erpobj_1 : B_{\vadcsp^\rho(u_0)} \to \real$ 
and the sample average function
$\hat{\erpobj}_{1,N} : B_{\vadcsp^\rho(u_0)} \to \real$
by
\begin{align}
\label{eq:erpobj_1}
\erpobj_1(u) = \cE{\rpobj_1(u,\xi)}
\quad \text{and} \quad 
\hat{\erpobj}_{1,N}(u) = 
\frac{1}{N} \sum_{i=1}^N\rpobj_1(u,\xi^i).
\end{align}
\begin{lemma}
\label{eq:erpobj1_smooth}
If \Cref{ass:pobj,ass:E,ass:existence,ass:pobjE} hold, then
$\erpobj_1$ 
and 
$\hat{\erpobj}_{1,N}$ 
are continuously differentiable
on $B_{\vadcsp^\rho(u_0)}$
and for each $u \in B_{\vadcsp^\rho(u_0)}$, we have
$\nabla \erpobj_1(u) = \cE{\nabla_u \rpobj_1(u, \xi)}$
and 
$\nabla \hat{\erpobj}_{1,N}(u) = (1/N)\sum_{i=1}^N
\nabla_u \rpobj_1(u, \xi^i)$.
\end{lemma}
We prove \Cref{eq:erpobj1_smooth} using 
\Cref{lem:rpobj_lipschitz}.
\begin{lemma}
\label{lem:rpobj_lipschitz}
If \Cref{ass:pobj,ass:E,ass:pobjE} hold,
then for all $\xi \in \Xi$, the function
$\rpobj_1(\cdot, \xi)$
is continuously differentiable on $\csp$, 
and for all  $u \in  B_{\vadcsp^\rho(u_0)}$, we have
$$
\rpobj_1(u,\xi)
\leq \rpobj_1(u_0,\xi) + C_K\zeta(\xi)
\norm[\csp]{u-u_0},
$$
where 
$C_K \in [0,\infty)$ is the 
operator norm of $\coperator$.
\end{lemma}%
\begin{proof}
Since $\coperator$ is linear and compact, it is bounded
\cite[Lem.\ 8.1-2]{Kreyszig1978}. Hence $C_K$ is finite.
For each $\xi \in \Xi$, $\rpobj_1(\cdot, \xi)$
is continuously differentiable on $\csp$
owing to the implicit function theorem
and \Cref{ass:pobj,ass:E}.
Since $\psi \geq 0$ and $\rpobj_1 \geq 0$, 
we have $u_0 \in \vadcsp^\rho(u_0)$. 
Using the mean-value theorem,
the convexity of $B_{\vadcsp^\rho(u_0)}$, 
$u$, $u_0 \in B_{\vadcsp^\rho(u_0)}$,
the formula \eqref{eq:lifted_rpobj_1}, 
and the estimate \eqref{eq:bound_Euz},  we obtain 
\begin{align*}
	\rpobj_1(u,\xi)-\rpobj_1(u_0,\xi)
	&\leq \sup_{t\in(0,1)}\, 
	\norm[\csp]{\nabla_u\rpobj_1(u_0+t(u-u_0),\xi)}\norm[\csp]{u-u_0}
	\\&\leq 
	C_K \zeta(\xi)
	\norm[\csp]{u-u_0}.
\end{align*}
\end{proof}%
\begin{proof}[{Proof of \Cref{eq:erpobj1_smooth}}]
Owing to $\psi(u_0) \in [0,\infty)$, $\cE{\rpobj(u_0,\xi)} < \infty$,
and $\rpobj_1 \geq 0$, we have $\cE{\rpobj_1(u_0,\xi)} \in [0,\infty)$.
Combined with  \Cref{lem:rpobj_lipschitz} and
$\cE{\zeta(\xi)} < \infty$, 
we find that $\erpobj_1$ is well-defined on
the open set $B_{\vadcsp^\rho(u_0)}$.
Moreover, $\rpobj_1(\cdot,\xi)$ is continuously differentiable
on $\csp$
for all $\xi \in \Xi$.
Combined with \Cref{ass:pobjE}
and \cite[Lem.\ C.3]{Geiersbach2020}, we find that
$\erpobj_1$ and $\hat{\erpobj}_{1,N}$ are \frechet\ differentiable
on $B_{\vadcsp^\rho(u_0)}$
with the asserted derivatives.
Using \Cref{ass:pobjE}
and the dominated convergence theorem, 
we obtain the  continuity of the \frechet\ derivatives
on $B_{\vadcsp^\rho(u_0)}$.
\end{proof}

\subsection{Compact Subsets}
\label{sec:compact_subsets}

We define a compact subset of the feasible set $\adcsp$
that contains the solutions
to the risk-neutral problem \eqref{eq:ocp} and eventually
those to its SAA problem
\eqref{eq:saa}.

Let us define 
\begin{align}
\label{eq:wad}
\wadcsp^\rho =
\vadcsp^\rho(u_0) \cap
\overline{\{\,\prox{\psi/\alpha}{-(1/\alpha) \coperator[v]} :\, \,
	v \in \bsp,\, 
	\norm[\bsp]{v} \leq \cE{\zeta(\xi)} + \rho
	\,\}}^{\norm[\csp]{\cdot}},
\end{align}
where $\overline{\csp_0}^{\norm[\csp]{\cdot}}$ denotes the 
$\norm[\csp]{\cdot}$-closure of $\csp_0 \subset \csp$.
\begin{lemma}
\label{lem:wadprecompact}
If \Cref{ass:pobj,ass:E,ass:pobjE} hold,
then  $\wadcsp^\rho$  is a compact subset of $\adcsp$.
\end{lemma}
\begin{proof}
We first show that the second set on the right-hand side
in \eqref{eq:wad} is compact.
\Cref{ass:pobjE_stability} yields
$\cE{\zeta(\xi)} < \infty$. Hence the set
$\{\, v \in \bsp : \, 
\norm[\bsp]{v} \leq \cE{\zeta(\xi)} + \rho  \,\}$
is bounded.
Thus, its image under the compact operator 
$\coperator$
(see \Cref{ass:pobjE_operator})
is precompact.
The operator 
$\prox{\psi/\alpha}{-(1/\alpha) \cdot} : \csp \to \csp$
is continuous, as $\prox{\psi/\alpha}{}$
is firmly nonexpansive  \cite[Prop.\ 12.28]{Bauschke2011}.
Since each continuous function maps precompact sets to precompact
ones \cite[p.\ 412]{Kreyszig1978}, 
the second set on the right-hand side
in \eqref{eq:wad} is compact.
This set is a subset of $\adcsp$
because $\prox{\psi/\alpha}{\csp} \subset \adcsp$.
Since $\vadcsp^\rho(u_0)$ is closed, 
the set $\wadcsp^\rho$ is compact.
Owing to $\vadcsp^\rho(u_0) \subset \adcsp$, 
we have $\wadcsp^\rho \subset \adcsp$.
\end{proof}

For each $\omega \in \Omega$, we define  
\begin{align}
\wadcsp^{[N]}(\omega) = 
\overline{\{\,\prox{\psi/\alpha}{-(1/\alpha)\nabla_u\hat{\erpobj}_{1,N}(u,\omega)}:
	u \in \vadcsp^\rho(u_0)\,\}}^{\norm[\csp]{\cdot}}.
\end{align}

\begin{lemma}
\label{lem:wadcontainssol}
Let \Cref{ass:pobj,ass:E,ass:existence,ass:pobjE} hold.
Then the following assertions hold.
\begin{enumerate}[nosep,font=\normalfont]
	\item 
	The set of solutions to  \eqref{eq:ocp} is contained
	in $\wadcsp^\rho$.
	
	\item 
	We have \wpone
	for all sufficiently large $N$,
	$\wadcsp^{[N]} \subset \wadcsp^\rho$.
\end{enumerate}
\end{lemma}%
\begin{proof}
\begin{enumerate}[nosep,wide]

	\item 
	\label{itm:proof:lem:wadcontainssol}
	Let $u^*$ be a solution to
	\eqref{eq:ocp}.
	Since $\rpobj_1 \geq 0$ and $\psi \geq 0$, we have
	$u^* \in \vadcsp^\rho(u_0)$.
	\Cref{eq:erpobj1_smooth} ensures that $\erpobj_1$ is continuously
	differentiable on $B_{\vadcsp^\rho(u_0)}$.
	Hence 
	$u^* = \prox{\psi/\alpha}{-(1/\alpha)\nabla \erpobj_1(u^*)}$
	(cf.\ \cite[Prop.\ 3.5]{Pieper2015} and \cite[p.\ 2092]{Mannel2020}).
	Using \Cref{ass:pobjE}, 
	in particular the bound in \eqref{eq:bound_Euz},
	and \cite[Prop.\ 1.2.2]{Hytoenen2016}, we have
	\begin{align*}
		\norm[\bsp]{\cE{M(u^*,\xi)}}
		\leq \cE{\norm[\bsp]{M(u^*,\xi)}}
		\leq \cE{\zeta(\xi)} < \infty.
	\end{align*}
	Combined with \eqref{eq:lifted_rpobj_1}
	and \cite[eq.\ (1.2)]{Hytoenen2016}, we find that
	$$
	\nabla \erpobj_1(u^*) 
	=\cE{\coperator M(u^*,\xi)}
	=\coperator\cE{M(u^*,\xi)}.
	$$
	Since 
	$u^* = \prox{\psi/\alpha}{-(1/\alpha)\coperator\cE{M(u^*,\xi)}}$
	and $\rho > 0$,
	we have $u^* \in \wadcsp^\rho$
	(see \eqref{eq:wad}).
	\item 
	The (strong)  law of large numbers 
	ensures 
	$(1/N) \sum_{i=1}^N\zeta(\xi^i) \to \cE{\zeta(\xi)}$
	\wpone as $N \to \infty$.
	Combined with $\rho > 0$, 
	we deduce the existence of an event $\Omega_1 \in \cF$
	with $P(\Omega_1) = 1$ and
	for each $\omega\in\Omega_1$, 
	there exists $n(\omega) \in \natural$
	such that for all $N \geq n(\omega)$, we have
	\begin{align*}
		\frac{1}{N} \sum_{i=1}^N\zeta(\xi^i(\omega))
		\leq  \cE{\zeta(\xi)}+\rho.
	\end{align*}
	Fix $\omega \in \Omega_1$ and let $N \geq n(\omega)$.
	Let $u \in \vadcsp^\rho(u_0)$ be arbitrary.
	Using $\vadcsp^\rho(u_0) \subset B_{\vadcsp^\rho(u_0)}$ and
	\Cref{ass:pobjE}, we find that
	\begin{align*}
		\cnorm[\bigg]{\bsp}{\frac{1}{N}\sum_{i=1}^N M(u,\xi^i(\omega))}
		&\leq \frac{1}{N}\sum_{i=1}^N\norm[\bsp]{M(u,\xi^i(\omega))}
		\leq \frac{1}{N} \sum_{i=1}^N\zeta(\xi^i(\omega))
		\\
		&\leq \cE{\zeta(\xi)}+\rho,
	\end{align*}
	where the right-hand side is independent of $u \in \vadcsp^\rho(u_0)$.
	Furthermore
	$$\nabla \hat{\erpobj}_{1,N}(u,\omega) 
	=(1/N) \sum_{i=1}^N  \nabla_u \rpobj_1(u,\xi^i(\omega))
	=\coperator \bigg( (1/N) \sum_{i=1}^N M(u,\xi^i(\omega)) \bigg).$$
	We conclude that 
	$\prox{\psi/\alpha}{-(1/\alpha)\nabla \hat{\erpobj}_{1,N}(u,\omega)}
	\in \wadcsp^\rho$
	for each $u \in \vadcsp^\rho(u_0)$.
	Since $\wadcsp^\rho$ is closed
	(see \Cref{lem:wadprecompact}), we have
	$\wadcsp^{[N]}(\omega) \subset \wadcsp^\rho$.
	Hence
	\begin{align*}
		\Omega_1 \subset 
		\{\, \omega \in \Omega \colon \,
		\exists\, n(\omega) \in \natural 
		\quad \forall\, N \geq n(\omega);
		\quad  \wadcsp^{[N]}(\omega) \subset \wadcsp^\rho
		\,\}.
	\end{align*}
	The set on the right-hand side is a subset of $\Omega$.
	Since  $\Omega_1 \in \cF$, $P(\Omega_1) = 1$ and 
	$(\Omega, \cF, P)$ is complete, 
	the set on the right-hand side is measurable
	and hence  occurs \wpone. 
\end{enumerate}
\end{proof}

To establish the measurability of the event
``for all sufficiently large $N$, 		$\wadcsp^{[N]} \subset \wadcsp^\rho$,''
we used the fact that  $(\Omega, \cF, P)$ is complete.
Since this event equals the limit inferior of the sequence
$(\{\omega\in\Omega:\wadcsp^{[N]}(\omega) \subset \wadcsp^\rho\})_N$, 
the measurability of the event would also be implied by that of
$\{\omega\in\Omega:\wadcsp^{[N]}(\omega) \subset \wadcsp^\rho\}$
for each $N \in \natural$ \cite[p.\  55]{Billingsley2012}.
This approach would require
us to show that $\{\omega\in\Omega:\wadcsp^{[N]}(\omega) \subset \wadcsp^\rho\}$
is measurable for each $N \in \natural$, which entails those of
$\{\omega\in\Omega:\wadcsp^{[N]}(\omega) \subset \wadcsp^\rho\}$
and $\wadcsp^{[N]}$.
Using \cite[Thm.\ 8.2.8]{Aubin2009}, we can show that
$\wadcsp^{[N]}$ is measurable. However, an application of
\cite[Thm.\ 8.2.8]{Aubin2009} requires $(\Omega, \cF, P)$ be complete.

\subsection{Consistency of SAA optimal values and SAA solutions}
\label{subsect:consistency_solutions}

We demonstrate the consistency of the SAA optimal value
and the SAA solutions. 
Let $\vartheta^*$ and $\mathscr{S}$ be the optimal value
and the set of solutions to \eqref{eq:ocp}, respectively.
Moreover, for each $\omega \in \Omega$, let
$\hat{\vartheta}_N^*(\omega)$ and $\hat{\mathscr{S}}_N(\omega)$ be the optimal value
and the set of solutions to the SAA problem \eqref{eq:saa}, respectively.

We define
the distance $\dist{u}{\mathscr{S}}$ from $u \in \hat{\mathscr{S}}_N(\omega)$ 
to $\mathscr{S}$ and
the deviation $\deviation{\hat{\mathscr{S}}_N(\omega)}{\mathscr{S}}$ 
between the sets 
$\hat{\mathscr{S}}_N(\omega)$ and 
$\mathscr{S}$ by
\begin{align*}
\dist{u}{\mathscr{S}} = \inf_{v\in \mathscr{S}}\, \norm[\csp]{u-v}
\quad \text{and} \quad 
\deviation{\hat{\mathscr{S}}_N(\omega)}{\mathscr{S}} = 
\sup_{u\in \hat{\mathscr{S}}_N(\omega)}\, 	\dist{u}{\mathscr{S}}. 
\end{align*}

\begin{Theorem}
\label{prop:deviation_solutions}
If \Cref{ass:pobj,ass:E,ass:existence,ass:pobjE} hold,
then
$\hat{\vartheta}_N^* \to \vartheta^*$
and
$\deviation{\hat{\mathscr{S}}_N}{\mathscr{S}} \to 0$
\wpone\	as $N \to \infty$.
\end{Theorem}

We prepare our proof of 
\Cref{prop:deviation_solutions},
which is based on that of \cite[Thm.\ 5.3]{Shapiro2014}.

\begin{lemma}
\label{lem:rpobjisrlsc}
If \Cref{ass:pobj,ass:E,ass:pobjE} hold,
then the  function $\rpobj_1$
defined in 	\eqref{eq:rpobj1}
is a \Caratheodory\ function
on $\wadcsp^\rho \times \Xi$. 
Moreover, $(\rpobj_1(u,\xi))_{u\in\wadcsp^\rho}$
is dominated by an integrable function.
\end{lemma}
\begin{proof}
\Cref{lem:wadprecompact} ensures that $\wadcsp^\rho$ 
is a compact metric space. 
Since $\wadcsp^\rho \subset \csp$
and $\rpobj_1$ is a \Caratheodory\ function  on $\csp \times \Xi$
(see \Cref{lem:rpobj1islsc}),
the function $\rpobj_1$ is a
\Caratheodory\ function on 
$\wadcsp^\rho \times \Xi$. 
\Cref{lem:rpobj_lipschitz} ensures that for all 
$u \in \wadcsp^\rho \subset \vadcsp^\rho(u_0)$,
\begin{align*}
	\rpobj_1(u,\xi) \leq 
	\rpobj_1(u_0,\xi) + C_K\zeta(\xi)
	\sup_{u\in\wadcsp^\rho}\, \norm[\csp]{u-u_0}.
\end{align*}
The  random variable  on the right-hand side is integrable owing to 
the integrability of $\zeta$ (see 	\Cref{ass:pobjE_stability}),
the boundedness of 
$\wadcsp^\rho$ (see \Cref{lem:wadprecompact}),
$C_K \in [0,\infty)$, and $\cE{\rpobj(u_0,\xi)} < \infty$.
Combined with $\rpobj_1 \geq 0$, we find that
$(\rpobj_1(u,\xi))_{\wadcsp^\rho}$  
is dominated by an integrable
random variable. 
\end{proof}

\begin{lemma}
\label{lem:varthetansn}
If \Cref{ass:pobj,ass:E,ass:existence,ass:pobjE} hold,
then  for each $N \in \natural$, the functions
$\hat{\vartheta}_N^*$ 
and 
$\deviation{\hat{\mathscr{S}}_N}{\mathscr{S}}$
are  measurable.
\end{lemma}

\begin{proof}
For each $\omega \in \Omega$, 
\Cref{ass:existence} ensures
that  $\hat{\mathscr{S}}_N(\omega)$ is nonempty.
The function $\rpobj_1$ is \Caratheodory\ function
on $\csp \times \Xi$	according to \Cref{lem:rpobjisrlsc} and $\psi$
is lower semicontinuous according to \Cref{ass:pobj_3}.
Hence $\hat{\vartheta}_N^*$  is 
measurable 	\cite[Lem.\ III.39]{Castaing1977}
and the set-valued mapping 
$\hat{\mathscr{S}}_N$ is measurable
\cite[p.\ 86]{Castaing1977}. 
\Cref{ass:existence} implies that
$\mathscr{S}$ is nonempty and, hence, 
$\dist{\cdot}{\mathscr{S}}$ is 
(Lipschitz) continuous
\cite[Thm.\ 3.16]{Aliprantis2006}.
For each $\omega \in \Omega$, 
$\hat{\erpobj}_N(\cdot, \omega)$ is lower semicontinuous
and hence $\hat{\mathscr{S}}_N(\omega)$ 
is closed. 
Thus
$\deviation{\hat{\mathscr{S}}_N}{\mathscr{S}} $ is measurable
\cite[Thm.\ 8.2.11]{Aubin2009}.
\end{proof}

\begin{lemma}
\label{lem:uslln}
If  \Cref{ass:pobj,ass:E,ass:existence,ass:pobjE} hold, 
then 
$\hat{\erpobj}_{N}$ 
converges to $\erpobj$
\wpone\ uniformly on $\wadcsp^\rho$.
\end{lemma}
\begin{proof}
We first verify the hypotheses of the uniform law of large numbers
established in  \cite[Cor.\ 4:1]{LeCam1953} to 
demonstrate the uniform
almost sure convergence of 
$\hat{\erpobj}_{1,N}$ 
to $\erpobj_1$
on $\wadcsp^\rho$.

\Cref{lem:rpobjisrlsc} ensures that
$\rpobj_1$ is a \Caratheodory\ function
on $\wadcsp^\rho \times \Xi$
and that $(\rpobj_1(u,\xi))_{u\in\wadcsp^\rho}$
is dominated by an integrable function.
Moreover, $\wadcsp^\rho$ is a  compact metric space
(see \Cref{lem:wadprecompact}).
Since $\xi^1, \xi^2, \ldots$ are independent
identically distributed random elements, the uniform law of large numbers
\cite[Cor.\ 4:1]{LeCam1953} implies
that
$\hat{\erpobj}_{1,N}(\cdot) = (1/N) \sum_{i=1}^N \rpobj_1(\cdot,\xi^i)$ 
converges to $\erpobj_1(\cdot) = \cE{\rpobj_1(\cdot,\xi)}$
\wpone\ uniformly on $\wadcsp^\rho$.

Since $\adcsp$ is the domain of $\psi$ and $\psi \geq 0$,
we have $ \psi(u)  \in [0,\infty)$
for all $u \in \adcsp$. \Cref{lem:wadprecompact} ensures
$\wadcsp^\rho \subset \adcsp$. 
Hence for all $u \in \wadcsp^\rho$,
\begin{align*}
	\hat{\erpobj}_N(u) - \erpobj(u)
	&= \hat{\erpobj}_{1,N}(u) +(\alpha/2)\norm[\csp]{u}^2 + \psi(u)
	- \big(\erpobj_1(u) + (\alpha/2)\norm[\csp]{u}^2 + \psi(u) \big)
	\\
	& = \hat{\erpobj}_{1,N}(u)  - \erpobj_1(u).
\end{align*}
Therefore, the assertion follows from the above uniform 
convergence statement.
\end{proof}

\Cref{lem:sninwadcp} demonstrates that
the SAA solution set is eventually contained in the
compact set $\wadcsp^\rho$.

\begin{lemma}
\label{lem:sninwadcp}
If \Cref{ass:pobj,ass:E,ass:existence,ass:pobjE} hold,
then \wpone\ for all sufficiently large $N$, 
$\hat{\mathscr{S}}_N \subset \wadcsp^\rho$.
\end{lemma}

\begin{proof}
First, we show that 
\wpone\ for all sufficiently large $N$, 
$\hat{\mathscr{S}}_N \subset \vadcsp^\rho(u_0)$.
\Cref{lem:rpobjisrlsc} ensures that $\rpobj_1$ is a \Caratheodory\
function on $\csp \times \Xi$.
Since $\rpobj \geq 0$, $u_0 \in \adcsp$,
and $\cE{\rpobj(u_0,\xi)} < \infty$, 
the (strong) law of large numbers ensures
\begin{align*}
	\frac{1}{N}\sum_{i=1}^N \rpobj(u_0,\xi^i)
	\to \cE{\rpobj(u_0,\xi)}
	\quad \text{\wpone} \quad \as \quad N \to \infty.
\end{align*}
Combined with $\rho > 0$, we deduce 
the existence of an event $\Omega_1 \in \cF$
such that $P(\Omega_1) = 1$ and for
each $\omega \in \Omega_1$, there exists
$n_1(\omega) \in \natural$ such that
for all $N \geq n_1(\omega)$, we have
\begin{align}
	\label{eq:sllnineqj}
	\frac{1}{N}\sum_{i=1}^N \rpobj(u_0,\xi^i(\omega))
	\leq  \cE{\rpobj(u_0,\xi)} + \rho.
\end{align}
Fix $\omega \in \Omega_1$ and let $N \geq n_1(\omega)$.
Using $\psi \geq 0$ and $\rpobj_1 \geq 0$, 
we have for all $u_N^* = u_N^*(\omega) \in \hat{\mathscr{S}}_N(\omega)$, 
\begin{align*}
	(\alpha/2) \norm[\csp]{u_N^*}^2 
	&\leq \frac{1}{N}\sum_{i=1}^N \rpobj_1(u_N^*,\xi^i(\omega))
	+ \psi(u_N^*) + (\alpha/2) \norm[\csp]{u_N^*}^2 
	\\
	&\leq \frac{1}{N}\sum_{i=1}^N \rpobj(u_0,\xi^i(\omega)).
\end{align*}
Combined with \eqref{eq:sllnineqj}, we find that
$\hat{\mathscr{S}}_N(\omega) \subset \vadcsp^\rho(u_0)$.

By construction of $\wadcsp^{[N]}$, 
we have
$\hat{\mathscr{S}}_N(\omega) \cap \vadcsp^\rho(u_0) \subset
\wadcsp^{[N]}(\omega)$
for all $\omega \in \Omega$. 
Indeed, 
if $u_N^*(\omega) \in \hat{\mathscr{S}}_N(\omega) \cap \vadcsp^\rho(u_0)$, 
then  we have the first-order optimality condition
$u_N^*(\omega) = \prox{\psi/\alpha}{-(1/\alpha)
	\nabla \hat{\erpobj}_{1,N}(u_N^*(\omega),\omega)}$.
Hence $u_N^*(\omega) \in \wadcsp^{[N]}(\omega)$.
\Cref{lem:wadcontainssol}
implies that 
\wpone for all sufficiently large 
$N$, $\wadcsp^{[N]} \subset \wadcsp^\rho$. 
Hence there exists $\Omega_2 \in \cF$
with $P(\Omega_2) = 1$
and for each $\omega \in \Omega_2$
there exists $n_2(\omega) \in \natural$
such that 
for all $N \geq n_2(\omega)$, 
$\wadcsp^{[N]}(\omega) \subset \wadcsp^\rho$.	
Putting together the pieces, we find that 
for all $\omega \in \Omega_1 \cap \Omega_2$
and each $N \geq \max\{n_1(\omega), n_2(\omega)\}$, 
we have 
$\hat{\mathscr{S}}_N(\omega) \subset \wadcsp^\rho$.
Since $(\Omega,\cF, P)$ is complete and
$P(\Omega_1 \cap \Omega_2) = 1$, we have
\wpone for all sufficiently large $N$,
$\hat{\mathscr{S}}_N \subset \wadcsp^\rho$.
\end{proof}

\begin{proof}[{Proof of \Cref{prop:deviation_solutions}}]
The proof is based on that of 
\cite[Thm.\ 5.3]{Shapiro2014}.
\Cref{lem:wadcontainssol} yields
$\mathscr{S} \subset \wadcsp^\rho$.
\Cref{lem:sninwadcp} ensures that \wpone 
for all sufficiently large $N$, 
$\hat{\mathscr{S}}_N \subset \wadcsp^\rho$.
Hence, we deduce the existence of
an event $\Omega_1 \in \cF$
with $P(\Omega_1) = 1$ 
and for each $\omega \in \Omega_1$ there exists
$n(\omega) \in \natural$ such that for all 
$N \geq n(\omega)$,
$\hat{\mathscr{S}}_N(\omega) \subset \wadcsp^\rho$.
\Cref{lem:uslln} ensures that
$\hat{\erpobj}_{N}(\cdot,\omega)$ 
converges to $\erpobj(\cdot)$
uniformly on $\wadcsp^\rho$ 
for almost all $\omega \in \Omega$.
Therefore, there exists $\Omega_2 \in \cF$
with $P(\Omega_2) = 1$
and for each $\omega \in \Omega_2$, 
$\hat{\erpobj}_{N}(\cdot,\omega)$ 
converges to $\erpobj(\cdot)$
uniformly on $\wadcsp^\rho$.

We show that $\hat{\vartheta}_N^*(\omega) \to \vartheta^*$  as
$N \to \infty$
for each $\omega \in \Omega_1 \cap \Omega_2$.
Fix $\omega \in \Omega_1 \cap \Omega_2$.
\Cref{ass:existence} ensures that 
$\mathscr{S}$ and
$\hat{\mathscr{S}}_N(\omega)$ are nonempty for all $N \in \natural$. 
Let $u^* \in \mathscr{S}$ 
and let $u_N^*(\omega) \in \hat{\mathscr{S}}_N(\omega)$.
Then for all $N \geq n(\omega)$, we have
$u_N^*(\omega) \in \wadcsp^\rho$
and hence 
$|\hat{\vartheta}_N^*(\omega)-\vartheta^*|
\leq \sup_{u \in \wadcsp^\rho}\,
|\hat{\erpobj}_{N}(u,\omega)-\erpobj(u)|
$
for all $N \geq n(\omega)$
(cf.\  \cite[pp.\ 194--195]{Kaniovski1995}).
We deduce $\hat{\vartheta}_N^*(\omega) \to \vartheta^*$  as
$N \to \infty$.

Next, we show that
$\deviation{\hat{\mathscr{S}}_N(\omega)}{\mathscr{S}} \to 0$
as $N \to \infty$
for each $\omega \in \Omega_1 \cap \Omega_2$.
Fix $\omega \in \Omega_1 \cap \Omega_2$.
Since $\mathscr{S}$ is nonempty
(see \Cref{ass:existence}), the function $\dist{\cdot}{\mathscr{S}}$
is (Lipschitz) continuous \cite[Thm.\ 3.16]{Aliprantis2006}.
For each $N \geq n(\omega)$, the set
$\hat{\mathscr{S}}_N(\omega)$ is closed and
$\hat{\mathscr{S}}_N(\omega) \subset \wadcsp^\rho$.
Hence $\hat{\mathscr{S}}_N(\omega)$ is compact for each 
$N \geq n(\omega)$.
Therefore, for each $N \geq n(\omega)$, 
there exists $u_N = u_N(\omega) \in 
\wadcsp^\rho$ with
$\dist{u_N}{\mathscr{S}} = 
\deviation{\hat{\mathscr{S}}_N(\omega)}{\mathscr{S}}$.	
Suppose that 
$\deviation{\hat{\mathscr{S}}_N(\omega)}{\mathscr{S}} \not \to 0$. 
We deduce the existence of 
a subsequence 
$\mathcal{N} = \mathcal{N}(\omega)$ of
$(n(\omega),n(\omega)+1, \ldots)$
such that
$\deviation{u_N}{\mathscr{S}} \geq \varepsilon$
for all $N \in \mathcal{N}$ and  some $\varepsilon > 0$,
and $u_N \to \bar{u} \in \wadcsp^\rho$
as $\mathcal{N} \ni N \to \infty$.
Combined with the fact that $\dist{\cdot}{\mathscr{S}}$ is continuous,
we obtain  $\bar{u} \not\in \mathscr{S}$. Hence $\erpobj(\bar{u}) > 
\vartheta^*$.
We have
\begin{align*}
	\liminf_{\mathcal{N}\ni N \to \infty} \,\hat{\erpobj}_{N}(u_N,\omega)
	= \lim_{\mathcal{N}\ni N \to \infty}
	\big(\hat{\erpobj}_{N}(u_N,\omega)-\erpobj(u_N)\big)
	+ 
	\liminf_{\mathcal{N}\ni N \to \infty}\, \erpobj(u_N).
\end{align*}
The uniform convergence implies that
the first term in the right-hand side is zero.
Since $\erpobj$ is lower semicontinuous on $\adcsp$
(see \Cref{ass:pobj,eq:erpobj1_smooth}),
$\erpobj(\bar{u}) > \vartheta^*$, 
and $\hat{\erpobj}_{N}(u_N,\omega) = \hat{\vartheta}_N^*(\omega)$,
we find that
\begin{align*}
	\liminf_{\mathcal{N}\ni N \to \infty} \hat{\vartheta}_N^*(\omega) = 
	\liminf_{\mathcal{N}\ni N \to \infty} \,\hat{\erpobj}_{N}(u_N,\omega)
	= \liminf_{\mathcal{N}\ni N \to \infty}\, \erpobj(u_N)
	\geq \erpobj(\bar{u}) > \vartheta^*.
\end{align*}
This contradicts $\hat{\vartheta}_N^*(\omega) \to \vartheta^*$
as $N \to \infty$.
Hence 	$\deviation{\hat{\mathscr{S}}_N(\omega)}{\mathscr{S}} \to 0$
as $N \to \infty$.
Combined with \Cref{lem:varthetansn}
and the fact that $P(\Omega_1 \cap \Omega_2) = 1$, 
we obtain the almost sure
convergence statements.
\end{proof}

\section{Examples}
\label{sec:applications}

We present three  risk-neutral nonlinear PDE-constrained
optimization problems 
and verify the assumptions made in \cref{sec:assumptions}, except
\Cref{ass:existence} on the existence of solutions in order to
keep the section relatively short.

We use the following facts.
(i)
The Sobolev spaces $H_0^1(\domain)$ and  $H^1(\domain)$ 
are separable Hilbert spaces \cite[Thm.\ 3.5]{Adams1975}.
(ii) If a real Banach space is reflexive and separable, 
then its dual is separable \cite[Thm.\ 1.14]{Adams1975}.
(iii)
The operator norm of a linear, bounded operator  equals that of
its (Banach space-)adjoint operator 
\cite[Thm.\ 4.5-2]{Kreyszig1978}.
(iv) If   $\Lambda_1$ and $\Lambda_2$ are real, reflexive Banach spaces
and $\Upsilon \colon \Lambda_1 \to \Lambda_2$ is linear and bounded, 
then $(\Upsilon^*)^* = \Upsilon$
\cite[p.\ 390]{Alt2016}
(see also \cite[Thm.\ 8.57]{Renardy2004})
because we write $(\Lambda_i^*)^* = \Lambda_i$ for $i \in \{1,2\}$.

\subsection{Boundary optimal control of a semilinear state equation}
\label{subsec:boundary}

We consider the risk-neutral boundary optimal control of a 
parameterized semilinear PDE.
Our model problem is based on the 
deterministic semilinear boundary control problems
studied in \cite{Troeltzsch2010a,Hinze2009,Casas2012c,Kahlbacher2008}.

We consider 
\begin{align}
\label{eq:glocp}
\min_{u\in L^2(\partial \domain)}\, (1/2)\cE{\norm[L^2(\domain)]{S(u,\xi)-y_d}^2}
+ (\alpha/2)\norm[L^2(\partial \domain)]{u}^2
+ \psi(u),
\end{align} 
where $\partial\domain$ is the boundary of $\domain \subset \real^{2}$
and 
for each $(u,\xi) \in L^2(\partial \domain) \times \Xi$, 
the state $S(u,\xi) \in H^1(\domain)$ is the weak solution to: find
$y \in H^1(\domain)$ with
\begin{align}
\label{eq:gleq}
-\nabla \cdot (\kappa(\xi)\nabla y) + g(\xi) y + y^3 = \rrhs(\xi)
\;\; \text{in} \;\;  \domain,
\quad 
\rdc(\xi) \partial_\nu y + \sigma(\xi)y
= Bu 
\;\;  \text{on} \;\;  \partial\domain,
\end{align}
where $\partial_\nu y$ is the normal derivative of $y$; see
\cite[p.\ 31]{Troeltzsch2010a}.
For a bounded Lipschitz domain $\domain \subset \real^d$, 
we denote by $L^2(\partial \domain)$
the space of square integrable functions on $\partial \domain$
and by $L^\infty(\partial \domain)$ that of essentially bounded functions
\cite[p.\ 263]{Alt2016}.
The space $L^2(\partial \domain)$ is a  Hilbert space
with inner product
$\inner[L^2(\partial\domain)]{v}{w} = \int_{\partial \domain} 
v(x)w(x) \du \mathrm{H}^{d-1}(x)$, where
$\mathrm{H}^{d-1}$ is the $(d-1)$-dimensional Hausdorff measure
on $\partial \domain$
\cite[Thm.\ 3.16 and pp.\ 47, 263 and 267]{Alt2016}.
The space $L^2(\partial \domain)$ is separable \cite[Thm.\ 4.1]{Necas2012}.

We formulate assumptions on the control problem \eqref{eq:glocp}.

\begin{itemize}
[wide,nosep,leftmargin=*]
\item 
$\domain \subset \real^2$
is a bounded Lipschitz domain.
\item 
\label{itm:random_data1:app1}
$\kappa$, $g : \Xi \to L^\infty(\domain)$ are strongly measurable
and there exist $ \kappa_{\min}$, $ \kappa_{\max}$, 
$g_{\min}$, $g_{\max} \in (0,\infty)$ such that
$ \kappa_{\min} \leq \kappa(\xi) \leq  \kappa_{\max}$
and $g_{\min}  \leq g(\xi) \leq  g_{\max}$
for all $\xi \in \Xi$.

\item 
$\rrhs \colon \Xi \to L^2(\domain)$
and $\sigma : \Xi \to L^\infty(\partial \domain)$
are strongly measurable
with $\cE{\norm[L^2(\domain)]{\rrhs(\xi)}^2}< \infty$,
$\cE{\norm[L^\infty(\partial \domain)]{\sigma(\xi)}^2} < \infty$
and $\sigma(\xi) \geq 0$
for all $\xi \in \Xi$.
\item 
$B : L^2(\partial \domain) \to L^2(\partial\domain)$
is a linear, bounded operator.

\item 
$y_d \in L^2(\domain)$, $\alpha > 0$,
and $\psi : L^2(\partial \domain) \to [0,\infty]$
is proper, convex, and lower semicontinuous.
\end{itemize}
Throughout the section, we assume these conditions be satisfied.

We establish \Cref{ass:pobj}.  
Since the embedding $H^1(\domain) \embedding L^2(\domain)$ is continuous,
the function $\pobj_1: H^1(\domain) \to [0,\infty)$ defined
by  $\pobj_1(y) = (1/2)\norm[L^2(\domain)]{y-y_d}^2$
is continuously differentiable. 
We find that \Cref{ass:pobj} holds true.

We formulate the weak form of \eqref{eq:gleq} as an operator equation; 
cf.\ \cite[eq.\ (2)]{Kahlbacher2008}.
We define 
$E: H^1(\domain) \times L^2(\partial \domain) \times \Xi \to H^1(\domain)^*$ by
\begin{align}
\label{eq:boundarypde}
\begin{aligned}
	\dualp[H^1(\domain)]{E(y,u,\xi)}{v}
	&= 
	\inner[L^2(\domain)^{2}]{\kappa(\xi)\nabla y}{\nabla v}
	+\inner[L^2(\domain)]{g(\xi)y+y^3}{v}
	\\
	&\quad+ \inner[L^2(\partial\domain)]{\sigma(\xi)\tau_{\partial\domain}[y]}
	{\tau_{\partial\domain}[v]}
	\\
	&\quad -\inner[L^2(\domain)]{\rrhs(\xi)}{v}
	-\inner[L^2(\partial\domain)]{Bu}{\tau_{\partial\domain}[v]}.
\end{aligned}
\end{align}
where $\tau_{\partial\domain} : H^1(\domain) \to L^2(\partial \domain)$
is the trace operator. We refer the reader to 
\cite[p.\ 268]{Alt2016} for the definition of $\tau_{\partial\domain}$.
Since $\domain$ has a Lipschitz boundary,
the trace  operator $\tau_{\partial\domain}$ is linear and
compact \cite[Thm.\ 6.2]{Necas2012}.

We verify \Cref{ass:E}.
Using \cite[Thm.\ 1.15]{Hinze2009}, we find that
$E(y,u,\xi) = 0$ has a unique solution $S(u,\xi) \in H^1(\domain)$
for each $(u,\xi) \in L^2(\partial \domain) \times \Xi$.
Since the embedding $H^1(\domain) \embedding L^6(\domain)$
is continuous \cite[Thm.\ 1.14]{Hinze2009}, we have
$y^3 \in L^2(\domain)$ for each $y \in H^1(\domain)$
\cite[p.\ 57]{Hinze2009}
and the mapping $L^6(\domain) \ni y \mapsto y^3 \in L^2(\domain)$
is continuously differentiable \cite[p.\ 76]{Hinze2009}.
We find that
$E(\cdot, \cdot, \xi)$ is continuously differentiable.
Now, the Lax--Milgram lemma can be used to show that
$E_y(S(u,\xi),u,\xi)$ has a bounded inverse.
We show that $E(y,u,\cdot)$ is measurable for each 
$(y,u) \in H^1(\domain) \times L^2(\partial \domain)$. 
Since  $H^1(\domain)^*$ is separable, 
it suffices to show that 
$\xi \mapsto \dualp[H^1(\domain)]{E(y,u,\xi)}{v}$ is measurable
for each fixed $(y,v,u) \in H^1(\domain)^2 \times L^2(\partial \domain)$
\cite[Thm.\ 1.1.6]{Hytoenen2016}.
We define $\phi: L^\infty(\domain) \to \real$ by
$\phi(\nu) = \inner[L^2(\domain)^{2}]{\nu \nabla y}{\nabla v}$.
\hoelder's inequality ensures that $\phi$ is (Lipschitz)
continuous. Since 
$\xi \mapsto \inner[L^2(\domain)^{2}]{\kappa(\xi)\nabla y}{\nabla v}$
is the composition of the continuous function $\phi$
with $\kappa$, it is measurable
\cite[Cor.\ 1.1.11]{Hytoenen2016}. Similar arguments can be used
to establish the measurability of the other terms in
\eqref{eq:boundarypde}. Hence \Cref{ass:E} holds true.

We establish \Cref{ass:pobjE}.
Fix $(u,\xi) \in L^2(\partial \domain) \times \Xi$.
Choosing $v = S(u,\xi)$ in \eqref{eq:boundarypde} and using
\begin{align*}
\inner[L^2(\domain)^{2}]{\kappa(\xi)\nabla y}{\nabla y}
+\inner[L^2(\domain)]{g(\xi)y}{y}
\geq \min\{\kappa_{\min},g_{\min}\} \norm[H^1(\domain)]{y}^2
\end{align*}
valid for all $y \in H^1(\domain)$, 
we obtain the stability estimate
\begin{align}
\label{eq:gleq_h1_state}
\min\{\kappa_{\min},g_{\min}\}
\norm[H^1(\domain)]{S(u,\xi)}
\leq \norm[L^2(\domain)]{\rrhs(\xi)} +
C_{\tau_{\partial\domain}}
\norm[L^2(\partial\domain)]{Bu},
\end{align}
where 
$C_{\tau_{\partial\domain}}$
is the operator norm of $\tau_{\partial\domain}$.
For each $(u,\xi) \in L^2(\partial\domain) \times \Xi$, 
let $z(u,\xi)$ be the unique solution to the adjoint equation:
find $z \in H^1(\domain)$ with
\begin{align*}
\inner[L^2(\domain)^{2}]{\kappa(\xi)\nabla z}{\nabla v}
&+\inner[L^2(\domain)]{g(\xi)z+3S(u,\xi)^2z}{v}
+ \inner[L^2(\partial\domain)]{\sigma(\xi)\tau_{\partial\domain}[z]}
{\tau_{\partial\domain}[v]}
\\
&
= -\inner[L^2(\domain)]{S(u,\xi)-y_d}{v}
\quad \text{for all} \quad v \in H^1(\domain);
\end{align*}
cf.\ \cite[eq.\ (4.54)]{Troeltzsch2010a}
and \cite[p.\ 729]{Kahlbacher2008}.
For its solution $z(u,\xi)$,
we obtain 
\begin{align}
\label{eq:gleq_h1_adjoint}
\min\{\kappa_{\min},g_{\min}\}
\norm[H^1(\domain)]{z(u,\xi)}
\leq
\norm[L^2(\domain)]{S(u,\xi)-y_d}.
\end{align}%
Since $\tau_{\partial\domain}^*$ is the adjoint operator
of $\tau_{\partial\domain}$, we have
for all $u \in L^2(\partial \domain)$ and $v \in H^1(\domain)$,
$\inner[L^2(\partial\domain)]{Bu}{\tau_{\partial\domain}[v]}
= \dualp[H^1(\domain)]{\tau_{\partial\domain}^*Bu}{v} 
$.
Combined with
$\tau_{\partial\domain} = (\tau_{\partial\domain}^*)^*$
and the identity $E_u(S(u,\xi),u,\xi) = -\tau_{\partial\domain}^*B$
(cf.\ \cite[p.\ 136]{Hinze2009}),
the gradient formula in  \eqref{eq:Euz} yields
\begin{align*}
\nabla \rpobj_1(u,\xi) = 
-B^*\tau_{\partial\domain}[z(u,\xi)].
\end{align*}
We choose $\coperator = -B^*\tau_{\partial\domain}$ and $M(u,\xi) = z(u,\xi)$.
The operator $\coperator : H^1(\domain) \to L^2(\partial \domain)$ is compact,
as $B$ is linear and bounded and $\tau_{\partial\domain}$ is linear and compact
\cite[Thm.\ 6.2]{Necas2012}. 
Using \cite[Thm.\ 8.2.9]{Aubin2009} and the measurability
of $S(u,\cdot)$ (see \Cref{lem:rpobj1islsc}), we can show that $z(u,\cdot)$
is measurable for all $u \in L^2(\partial\domain)$.
The implicit function theorem implies that $z(\cdot, \xi)$ is continuous
for each $\xi \in \Xi$. 
Since $\psi$ is proper, there exists $u_0 \in L^2(\partial \domain)$
with $\psi(u_0) < \infty$.
Using Young's inequality, we have
$
\rpobj_1(u_0,\xi)
\leq \norm[L^2(\domain)]{y_d}^2+
\norm[L^2(\domain)]{S(u_0,\xi)}^2
$.
Combined with \eqref{eq:gleq_h1_state}, we find that
$\cE{\rpobj_1(u_0,\xi)} < \infty$ and hence $\cE{\rpobj(u_0,\xi)} < \infty$.
Let $B_{\vadcsp^\rho(u_0)}$ be an open, bounded ball about zero 
containing $\vadcsp^\rho(u_0)$
and let  $\radius$ be its radius. We define
\begin{align*}
\zeta(\xi) = 
\tfrac{1}{\min\{\kappa_{\min},g_{\min}\}}
\Big(
\norm[L^2(\domain)]{y_d}
+ \tfrac{C_{\tau_{\partial\domain}}
	C_{B}\radius + \norm[L^2(\domain)]{\rrhs(\xi)}}
{\min\{\kappa_{\min},g_{\min}\}}
\Big).
\end{align*}
where $C_B > 0$ is the operator norm of $B$. The random
variable $\zeta$ is integrable.
Using the stability estimates \eqref{eq:gleq_h1_state} 
and \eqref{eq:gleq_h1_adjoint}, we conclude that \Cref{ass:pobjE} holds true.

\subsection{Distributed control of a steady Burgers' equation}

We consider the risk-neutral distributed optimization
of a steady Burgers' equation.
Deterministic optimal control problems with the
Burgers' equation are studied, for example,
in 
\cite{delosReyes2004,Volkwein1997,Volkwein2000a,Volkwein2000}.
We refer the reader to
\cite{Kouri2014a,Kouri2013,Kouri2016,Kouri2020a}
for risk-neutral and risk-averse control
of the steady Burgers' equation.

Let us consider 
\begin{align}
\label{eq:sbocp}
\min_{u\in \adcsp}\, (1/2)\cE{\norm[L^2(0,1)]{S(u,\xi)-y_d}^2}
+ (\alpha/2)\norm[L^2(\domain_0)]{u}^2,
\end{align}
where $\domain_0 \subset (0,1)$ is a nonempty domain
and
for all $(u,\xi) \in L^2(\domain_0) \times \Xi$, 
the state $S(u,\xi) \in H^1(0,1)$
is the weak solution to the
steady Burgers' equation: find $y \in H_0^1(0,1)$ with
\begin{align*}
-\rdc(\xi)y'' + y y'
= \rrhs(\xi) + Bu\quad \text{in}  \quad (0,1),
\quad 
y(0) = 0, \;\; y(1) = 0,
\end{align*}
where $\rrhs :  \Xi \to L^2(\domain)$ and
$\rdc : \Xi \to (0,\infty)$.
As in \cite[p.\ 78]{Volkwein1997}, 
$B : L^2(\domain_0) \to L^2(0,1)$
is defined by $(Bu)(x) = u(x)$ if $x \in \domain_0$
and $0$ else. We consider homogeneous Dirichlet boundary conditions, 
as it simplifies the derivation of a state stability estimate.

The weak form of the steady Burgers' equation
has at least one solution $S(u,\xi) \in H_0^1(0,1)$
for each $(u,\xi) \in L^2(\domain_0) \times \Xi$
\cite[Prop.\ 3.1]{Volkwein2000a}.
We assume that the solution $S(u,\xi)$ be unique
to ensure that the reduced formulation
\eqref{eq:sbocp} is well-defined.
A condition sufficient  for uniqueness is that $\kappa(\xi)$
is sufficiently large \cite[Prop.\ 3.1]{Volkwein2000a}. 
We formulate the uniqueness as an assumption.

\begin{itemize}
[wide,nosep,leftmargin=*]
\item 
$\kappa \colon \Xi \to \real$ is measurable
and there exists $\kappa_{\min}$, $\kappa_{\max} \in (0,\infty)$ such that
$\kappa_{\min}\leq \kappa(\xi) \leq \kappa_{\max}$
for all $\xi \in \Xi$.

\item 
$\rrhs : \Xi \to L^2(0,1)$ is strongly measurable
and there exists
$\rrhs_{\max} \in (0,\infty)$
such that $\norm[L^2(0,1)]{\rrhs(\xi)} \leq \rrhs_{\max}$
for all $\xi \in \Xi$.
\item 
For each $(u,\xi) \in L^2(\domain_0) \times \Xi$, 
the solution $S(u,\xi) \in H_0^1(0,1)$ to
the weak form of the steady Burgers' equation is unique.

\item 
$y_d \in L^2(0,1)$, $\adcsp \subset L^2(\domain_0)$ 
is nonempty, closed, and convex, and
$\alpha > 0$.
\end{itemize}
Throughout the section, we assume these conditions be satisfied.

Let us verify \Cref{ass:pobj}.
The constraints in \eqref{eq:sbocp} can be modeled using
the indicator function
$\psi = I_{\adcsp}$. Since $\adcsp$ is nonempty, closed,
and convex, the function $I_{\adcsp}$ is proper, convex, 
and lower semicontinuous \cite[Ex.\ 2.67]{Bonnans2013}.
The function $\pobj_1 : H_0^1(\domain) \to [0,\infty)$
defined by 
$\pobj_1(y)  = (1/2)\norm[L^2(\domain)]{y-y_d}^2$
is continuously differentiable. 
Putting together the pieces, we find that
\Cref{ass:pobj} holds true.

We define 
$E : H_0^1(0,1) \times L^2(\domain_0) \times \Xi 
\to H^{-1}(0,1)$
by
\begin{align*}
\begin{aligned}
	\dualpHzeroone[0,1]{E(y, u, \xi)}{v}
	&= \inner[L^2(0,1)]{\rdc(\xi)y'}{v'}
	+ \inner[L^2(0,1)]{yy'}{v}\\
	& \quad - \inner[L^2(0,1)]{\rrhs(\xi)}{v} 
	- \inner[L^2(0,1)]{Bu}{v}.
\end{aligned}
\end{align*}
Let $\iota : H_0^1(0,1) \to L^2(0,1)$
be the embedding operator of the compact embedding
$H_0^1(\domain) \embedding L^2(0,1)$.
We have $\dualpHzeroone[0,1]{\iota^*[Bu]}{v} = \inner[L^2(0,1)]{Bu}{v}$
for all $v \in H_0^1(\domain)$ and $u \in L^2(\domain_0)$.

We show that \Cref{ass:E} holds true.
The operator $E$ is well-defined \cite[pp.\ 76 and 80]{Volkwein1997}
and
$E(\cdot,\cdot,\xi)$ is twice continuously differentiable
for each $\xi \in \Xi$ 	\cite[p.\ 81]{Volkwein1997}.
For each $(u,\xi) \in L^2(\domain_0) \times \Xi$, 
$E_y(S(u,\xi),u,\xi)$ has a bounded inverse
\cite[p.\ A1866]{Kouri2013}.
Using  arguments similar to those in  \cref{subsec:boundary},
we can show that $E(y,u,\cdot)$ is measurable
for each $(y,u) \in H_0^1(\domain) \times L^2(\domain_0)$.
We conclude that \Cref{ass:E} holds true.

Using the gradient
formula \eqref{eq:Euz},
$(\iota^*)^* = \iota$,
and $E_u(S(u,\xi),u,\xi) = -\iota^*B$,
we find that
\begin{align}
\label{eq:sburgers_nablarobj1}
\nabla \rpobj_1(u,\xi) = -B^*\iota[z(u,\xi)],
\end{align}
where for each $(u,\xi) \in L^2(\domain_0) \times \Xi$, 
$z(u,\xi) \in H_0^1(0,1)$ solves
the adjoint equation: find $z \in H_0^1(0,1)$
with
\begin{align*}
\kappa(\xi)\inner[L^2(0,1)]{z'}{v'}
-\inner[L^2(0,1)]{S(u,\xi)z'}{v}
=
-\inner[L^2(0,1)]{S(u,\xi)-y_d}{v}
\end{align*}
for all $v \in H_0^1(0,1)$;
cf.\ \cite[205--206]{delosReyes2004} and \cite[p.\ 83]{Volkwein1997}.
Since $\iota$ is linear and compact, and $B$ is linear and bounded,
the operator $\coperator = -B^* \iota$ is compact
\cite[Thm.\ 8.2-5 and p.\ 427]{Kreyszig1978}.
We choose $M(u,\xi) = z(u,\xi)$. 

We establish \Cref{ass:pobjE}.
Using \cite[Thm.\ 8.2.9]{Aubin2009} and the measurability
of $S(u,\cdot)$ (see \Cref{lem:rpobj1islsc}), we can show that $z(u,\cdot)$
is measurable for all $u \in L^2(\domain_0)$.
The implicit function theorem can be used to show that
$z(\cdot, \xi)$ is continuous. Hence $z$ is a \Caratheodory\ mapping.
Next, we derive an $H_0^1(0,1)$-stability estimate
for the state. We have $\norm[L^2(0,1)]{Bu} \leq \norm[L^2(\domain_0)]{u}$
for all $u \in L^2(\domain_0)$.
Hence the operator norm of $B$ is less than or equal to one.
We have
$\norm[L^p(0,1)]{v} \leq \norm[H_0^1(0,1)]{v}$
for each $v \in H_0^1(0,1)$
and $1\leq p \leq \infty$ \cite[Lem.\ 3.4 on p.\ 9]{Volkwein1997}.
Hence
\friedrichs' constant $C_\domain$
satisfies $C_\domain \leq 1$.
Using integration by parts, we
have $\inner[L^2(0,1)]{yy'}{y} = 0$  for  all $y \in H_0^1(0,1)$
\cite[p.\ 72]{Volkwein1997}. Choosing $v = S(u,\xi)$ in 
the weak form of Burgers' equation, we obtain
\begin{align}
\label{eq:nsburgers}
\kappa_{\min} \norm[H_0^1(0,1)]{S(u,\xi)}
\leq \norm[L^2(0,1)]{\rrhs(\xi)} + \norm[L^2(\domain_0)]{u};
\end{align}
cf.\ \cite[p.\ 75]{Volkwein1997}. Next, we establish a stability estimate
for $M(u,\xi) = z(u,\xi)$.
Combining the $L^\infty(0,1)$-stability estimate
established in  \cite[Lem.\ 3.4 on p.\ 83]{Volkwein1997}
with
$(1+\eu^{2x})\eu^{x} \leq 2\eu^{3x}$
valid for all $x \geq 0$, we obtain
\begin{align}
\label{eq:sburgers_adjoint}
\norm[L^\infty(0,1)]{z(u,\xi)}
\leq 2\kappa(\xi)^{-1}
\eu^{3\kappa(\xi)^{-1}\norm[L^1(0,1)]{S(u,\xi)}}
\norm[L^2(0,1)]{S(u,\xi)-y_d}.
\end{align}
Choosing $v = z(u,\xi)$ in the adjoint equation and
using the H\"older and \friedrichs inequalities,
and $C_\domain \leq 1$, 
we obtain
\begin{align}
\label{eq:sburgers_adjoint'}
\kappa(\xi)\norm[H_0^1(0,1)]{z(u,\xi)}
\leq 
\norm[L^2(0,1)]{S(u,\xi)} 
\norm[L^\infty(0,1)]{z(u,\xi)}
+\norm[L^2(0,1)]{S(u,\xi)-y_d}.
\end{align}
Since $\adcsp$ is nonempty, there exists $u_0 \in \adcsp$. 
Combined with \eqref{eq:nsburgers} and the definition of
$\pobj_1$, 
we find that $\cE{\rpobj(u_0,\xi)} < \infty$.
Let $B_{\vadcsp^\rho(u_0)}$ be an open, bounded ball 
about zero 
containing $\vadcsp^\rho(u_0)$
with radius $\radius > 0$. We define
$\zeta_1(\xi) = (1/\kappa_{\min})\big(\norm[L^2(0,1)]{\rrhs(\xi)} + \radius + 
\norm[L^2(0,1)]{y_d}\big)$ and
\begin{align*}
\zeta(\xi) = (1/\kappa_{\min})\zeta_1(\xi)\big((2/\kappa_{\min})\zeta_1(\xi)
\eu^{(3/\kappa_{\min})\zeta_1(\xi)}
+1\big).
\end{align*}
Combining
\eqref{eq:sburgers_nablarobj1}
and the stability estimates
\eqref{eq:nsburgers},
\eqref{eq:sburgers_adjoint} and
\eqref{eq:sburgers_adjoint'},
we conclude that \Cref{ass:pobjE} holds true
with $\zeta$ being an essentially bounded random variable.

\subsection{Distributed control of a semilinear state equation}

We consider a distributed control problem
with  a semilinear state
equation based on those considered in  \cite[Sect.\ 5]{Kouri2020}
and \cite[Sect.\ 5.2]{Kouri2019a}.
Risk-neutral optimization of semilinear PDEs 
are also studied, 
for example, in \cite{Garreis2019a,Geiersbach2020}.
We refer the reader to
\cite[Chap.\ 9]{Ulbrich2011}
and
\cite[Chap.\ 4]{Troeltzsch2010a}
for the analysis of deterministic, distributed
control problems with semilinear PDEs.

We consider
\begin{align}
\label{eq:secop}
\min_{u\in \adcsp}\, 
(1/2)\cE{\norm[L^2(\domain)]{\maxo{1-S(u,\xi)}}^2}
+(\alpha/2)\norm[L^2(\domain)]{u}^2,
\end{align}
where  $\maxo{\cdot} = \max\{0,\cdot\}$,
$\alpha > 0$, 
and $\adcsp \subset L^2(\domain)$ is a nonempty, closed, and convex.
For each $(u,\xi) \in L^2(\domain) \times \Xi$,  $S(u,\xi) \in H^1(\domain)$
is the  solution to: find $y \in H^1(\domain)$ with
$E(y,u,\xi) = 0$, where
the operator
$E: H^1(\domain) \times L^2(\domain) \times \Xi \to H^1(\domain)^*$ 
is defined by  
\begin{align}
\label{eq:sese}
\begin{aligned}
	\dualp[H^1(\domain)]{E(y,u,\xi)}{v} = 
	&\inner[L^2(\domain)^{2}]{\rdc(\xi)\nabla y}{\nabla v}
	+\inner[L^2(\domain)]{g(\xi)y+y^3}{v} 
	\\
	& \quad - 
	\inner[L^2(\domain)]{B(\xi)[u]}{v}
	-\inner[L^2(\domain)]{\rrhs(\xi)}{v}.
\end{aligned}
\end{align}
Let $\iota : H^1(\domain) \to L^2(\domain)$
be the compact embedding operator of the compact embedding
$ H^1(\domain) \embedding  L^2(\domain)$
\cite[Thm.\ 1.14]{Hinze2009}.
For each $\xi \in \Xi$,
we define $B(\xi) = \iota \widetilde{B}(\xi) \iota^*$.
The operator 
$\widetilde{B}(\xi) : H^1(\domain)^* \to H^1(\domain)$ 
is the solution operator to a  parameterized PDE\@.
For each 
$(f,\xi) \in H^1(\domain)^* \times \Xi$, 
$\widetilde{B}(\xi)f \in H^1(\domain)$ is the solution to: 
find $w \in H^1(\domain)$ with 
\begin{align}
\label{eq:seseb}
\inner[L^2(\domain)^{2}]{r(\xi)\nabla w}{\nabla v}
+ \inner[L^2(\domain)]{w}{v}
= \dualp[H^1(\domain)]{f}{v}
\quad \text{for all} \quad v \in H^1(\domain).
\end{align}
Since  the embedding $H^1(\domain) \embedding L^2(\domain)$ is continuous,
the  operator $\iota^*$ is given
by  $\dualp[H^1(\domain)]{\iota^*[u]}{v} = \inner[L^2(\domain)]{u}{v}$
for all $(u,v) \in L^2(\domain) \times H^1(\domain)$
\cite[p.\ 21]{Bonnans2013}.

The assumptions stated next
ensure the existence and uniqueness of
solutions to  the PDE defined by the operator in \eqref{eq:sese} and 
the well-posedness of the operator $\widetilde{B}(\xi)$;
see \cite[Sects.\  3 and 5]{Kouri2020}.

\begin{itemize}[wide,nosep,leftmargin=*]
\item $\domain \subset \real^2$ is a bounded Lipschitz  domain.

\item 	$\kappa$, $g : \Xi \to L^\infty(\domain)$ are strongly measurable
and there exist $ \kappa_{\min}$, $ \kappa_{\max}$, 
$g_{\min}$, $g_{\max} \in (0,\infty)$ such that
$\kappa_{\min} \leq \kappa(\xi) \leq  \kappa_{\max}$
and $g_{\min}  \leq g(\xi) \leq  g_{\max}$
for all $\xi \in \Xi$.

\item 
$\rrhs \colon \Xi \to L^2(\domain)$ and 
$r \colon \Xi \to L^\infty(\domain)$ are strongly measurable
and there exist $\rrhs_{\max}$, $r_{\min}$, $r_{\max} \in (0,\infty)$
such that
$\norm[L^2(\domain)]{\rrhs(\xi)} \leq \rrhs_{\max}$ and
$r_{\min} \leq r(\xi) \leq   r_{\max}$ 
for all $\xi \in \Xi$.
\end{itemize}
Throughout the section, we assume these conditions be satisfied.

\Cref{ass:pobj} is fulfilled since the function
$\pobj_1 : H^1(\domain) \to [0,\infty)$ defined by
$\pobj_1(y) = (1/2)\norm[L^2(\domain)]{\maxo{1-\iota y}}^2$ is 
continuously differentiable \cite[p.\ 14]{Kouri2020}.
We have
$\Du_y \pobj_1(y) = -\iota^*\maxo{1-\iota [y]}$.
Since $\iota[y] = y$,
we have for all $y \in H^1(\domain)$,
\begin{align}
\label{eq:app3:dyj}
\norm[H^1(\domain)^*]{\Du_y \pobj_1(y)}
\leq \norm[L^2(\domain)]{\maxo{1-y}}
\leq \norm[L^2(\domain)]{1}+ \norm[H^1(\domain)]{y}.
\end{align}

For each $\xi \in \Xi$, 
the operator $E(\cdot,\cdot,\xi)$ is continuously differentiable
\cite[p.\ 14]{Kouri2020} 
and 
for each $(u,\xi) \in L^2(\domain) \times \Xi$, 
$E_y(S(u,\xi),u,\xi)$ has a bounded inverse
\cite[p.\ 9]{Kouri2020}.
Using arguments similar to those in \cref{subsec:boundary},
we can show that $E(y,u,\cdot)$ is measurable
for each $(y,u) \in H^1(\domain) \times L^2(\domain)$.
We find that \Cref{ass:E} holds true. 

We verify \Cref{ass:pobjE}.
For each $(u,\xi) \in L^2(\domain) \times \Xi$, 
the adjoint state $z(u,\xi) \in H^1(\domain)$
is the solution to: find $z \in H^1(\domain)$ with
\begin{align*}
\inner[L^2(\domain)^{2}]{\kappa(\xi)\nabla z}{\nabla v}
+\inner[L^2(\domain)]{g(\xi)z+3S(u,\xi)^2z}{v}
= \inner[L^2(\domain)]{\maxo{1-S(u,\xi)}}{v}
\end{align*}
for all $v \in H^1(\domain)$.
Choosing $v = z(u,\xi)$ and  using
\eqref{eq:app3:dyj}, we obtain the stability estimate
\begin{align}
\label{eq:sese_h1_adjoint}
\min\{\kappa_{\min},g_{\min}\}
\norm[H^1(\domain)]{z(u,\xi)} \leq 
\norm[L^2(\domain)]{1} + \norm[H^1(\domain)]{S(u,\xi)}.
\end{align}
Moreover, for all $f \in H^1(\domain)$
and $u \in L^2(\domain)$, we have
the stability estimates
(cf.\  \cite[Sects.\ 3 and 5]{Kouri2020})
\begin{align}
\label{eq:sese_h1_state}
\begin{aligned}
	\min\{r_{\min},1\}
	\norm[H^1(\domain)]{\widetilde{B}(\xi)f} 
	&\leq 	\norm[H^1(\domain)^*]{f},
	\\
	\min\{\kappa_{\min},g_{\min}\}\norm[H^1(\domain)]{S(u,\xi)} 
	& \leq 
	\norm[H^1(\domain)^*]{\iota^*[B(\xi)u]+\iota^*[b(\xi)]}.
\end{aligned}
\end{align}
Using calculus for adjoint operators
\cite[p.\ 235]{Kreyszig1978}
and $\iota = (\iota^*)^*$, we find
that $(\iota^*B(\xi))^* = \iota \widetilde{B}(\xi)^* \iota^*\iota$.
Consequently, the gradient formula \eqref{eq:Euz} yields
\begin{align*}
\nabla \rpobj_1(u,\xi) = 
-\iota[\widetilde{B}(\xi)^*\iota^*\iota z(u,\xi)].
\end{align*}
We choose $K = - \iota$ and
$M(u,\xi) = \widetilde{B}(\xi)^*\iota^*\iota z(u,\xi)$.
Using the implicit function theorem and \cite[Prop.\ 4.3]{Kouri2020}, 
we find that $z$ is a \Caratheodory\ mapping.
Combined with \eqref{eq:sese_h1_state}, we obtain that
$M(\cdot, \xi)$ is continuous for each $\xi \in \Xi$.
Fix $f$, $v \in H^1(\domain)^*$. Using 
\cite[Thm.\ 8.2.9]{Aubin2009}, we can show that
$\xi \mapsto \widetilde{B}(\xi)f$
is  measurable.
Hence 
$\xi \mapsto \dualp[H^1(\domain)]{v}{\widetilde{B}(\xi)f}$
is measurable
\cite[Thm.\ 1.1.6]{Hytoenen2016}.
Since $\dualpstar[H^1(\domain)]{H^1(\domain)^*}{\widetilde{B}(\xi)^*v}{f}
= 
\dualp[H^1(\domain)]{v}{\widetilde{B}(\xi)f}$
for all $\xi \in \Xi$,
the mapping 
$\xi \mapsto \widetilde{B}(\xi)^*v$ is measurable
\cite[Thm.\ 1.1.6]{Hytoenen2016}. Since 
$H^1(\domain)$ is separable,
$\xi \mapsto \widetilde{B}(\xi)^*$ is strongly measurable
\cite[Thm.\ 1.1.6]{Hytoenen2016}.
Combined with the composition rules
\cite[Prop.\ 1.1.28 and Cor.\ 1.1.29]{Hytoenen2016}, we can show that
$M(u,\cdot)$ is measurable.

Using \eqref{eq:sese_h1_state} and the fact that $\adcsp$ is nonempty, 
we find that there exists $u_0 \in \adcsp$ with
$\cE{\rpobj(u_0,\xi)} < \infty$. 
Let $B_{\vadcsp^\rho(u_0)}$ be an open, bounded
ball about zero containing $\vadcsp^\rho(u_0)$
and let $\radius$ be its radius. We define the random variable
\begin{align*}
\zeta(\xi) = 
\tfrac{1}{\min\{\kappa_{\min},g_{\min}\}}
\bigg(\norm[L^2(\domain)]{1} + \tfrac{\norm[L^2(\domain)]{\rrhs(\xi)}
	+\tfrac{1}{\min\{r_{\min},1\}}\radius}
{\min\{\kappa_{\min},g_{\min}\}}\bigg).
\end{align*}
Our assumptions  and \hoelder's inequality
ensure that $\zeta$ is integrable.

Combined with the stability estimates
\eqref{eq:sese_h1_adjoint} and 
\eqref{eq:sese_h1_state}, we conclude that
\Cref{ass:pobjE} holds true.

\section{Discussion}
\label{sec:conclusion}

The analysis of the SAA approach for PDE-constrained optimization
under uncertainty 
is complicated by the fact that the feasible sets are generally noncompact, 
stopping us from directly applying the
consistency results developed in the literature
on M-estimation and stochastic programming. 
Inspired by the consistency results in \cite{Shapiro2003,Shapiro2014},
we constructed compact subsets of the feasible set 
that contain the solutions to the stochastic programs
and eventually those to their SAA problems, allowing us to establish
consistency of the SAA optimal values and SAA solutions.
To construct such compact sets, we combined 
the adjoint approach, optimality conditions, 
and PDE stability estimates. We applied our framework to three
risk-neutral nonlinear PDE-constrained optimization problems.

We comment on four limitations of our approach. 
First, our construction of the compact sets exploits 
the positivity of the regularization 
parameter $\alpha$, limiting our approach
at first to PDE-constrained optimization
problems with strongly convex control regularization. 
However, 
we can add $(\alpha/2)\norm[L^2(\domain)]{\cdot}^2$
with $\alpha > 0$ to the objective function, allowing us to establish
the consistency of regularized SAA solutions. If $\adcsp$ is contained
in a ball with radius $r_{\text{ad}} > 0$, $\varepsilon > 0$, and 
$\alpha = 2\varepsilon/r_{\text{ad}}^2$, then 
solutions to the regularized SAA problem
provide $\varepsilon$-optimal solutions to the non-regularized 
SAA problem \eqref{eq:saa}.\footnote{%
A point $x \in X$ is an $\varepsilon$-optimal solution to
$\inf_{x \in X}\, f(x)$ if $f(x) \leq \inf_{x \in X}\, f(x) + \varepsilon$.}
Second, the analysis developed here demonstrates the
consistency of SAA optimal values and SAA optimal solutions, but not of
SAA critical points. Since the risk-neutral PDE-constrained optimization
problems considered here are generally nonconvex, 
a consistency analysis of SAA critical points would be desirable. 
However, even though risk-neutral nonlinear
PDE-constrained optimization problems and their SAA problems
are generally nonconvex, significant progress has been made in
establishing convexity properties of nonlinear PDE-constrained
optimization problems \cite{Gahururu2019,Hintermueller2020} and 
in developing verifiable conditions that can be used to
certify global optimality of critical points
\cite{AhmadAli2016}.
Third, the construction of the compacts subsets
performed in \Cref{sec:assumptions} 
exploits the fact that the feasible set \eqref{eq:adcsp}
of the SAA problems is the same as that of the risk-neutral problem.
Therefore, our approach does not allow 
for a consistency analysis for SAA problems defined by random
constraints, such as those
resulting from sample-based approximations of expectation constraints
\cite[pp.\ 168--170]{Shapiro2014}.
Fourth, our analysis  does not
apply to risk-averse PDE-constrained optimization problems, as it exploits
smoothness of the expectation function.
However, our approach may be generalized to allow for the  consistency
analysis
of risk-averse PDE-constrained programs, such as those defined via the 
superquantile/conditional
value-at-risk \cite{Kouri2016}.

\begin{appendix}
\section{Lack of inf-compactness}
\label{sec:saalevelsets}
Besides the feasible set's lack of compactness, 
the set of $\varepsilon$-optimal solutions to the SAA problem
and the level sets of the SAA problem's objective function
may be noncompact for risk-neutral
PDE-constrained optimization
problems. 
We illustrate this observation on 
the semilinear PDE-constrained
problem
\begin{align}
	\label{eq:intro:ocp}
	\min_{u\in \adcsp}\,
	(1/2)\cE{\norm[L^2(\domain)]{S(u,\xi)}^2}
	+(\alpha/2)\norm[L^2(\domain)]{u}^2,
\end{align}
where $\alpha > 0$, 
$\adcsp = \{\, u \in L^2(\domain) \colon \norm[L^2(\domain)]{u} \leq 2\,\}$,
$\domain \subset \real^2$ is a bounded
Lipschitz domain, and
$\Xi$ is as in \cref{sec:assumptions}. 
For each $(u,\xi) \in L^2(\domain) \times \Xi$, the state
$S(u,\xi) \in H_0^1(\domain)$ is the  solution to 
the semilinear PDE: find $y \in H_0^1(\domain)$ with
\begin{align}
	\label{eq:intro:sl}
	\inner[L^2(\domain)^2]{\rdc(\xi)\nabla y}{\nabla v}
	+\inner[L^2(\domain)]{y^3}{v} =
	\inner[L^2(\domain)]{u}{v}
	\;\; \text{for all} \;\; v \in H_0^1(\domain).
\end{align}
We assume that $\rdc \colon \Xi \to L^\infty(\domain)$
is strongly measurable and that there exists $\rdc_{\min} > 0$
with $\rdc(\xi)\geq \rdc_{\min}$ for all $\xi \in \Xi$.
The SAA problem of \eqref{eq:intro:ocp} is given by
\begin{align}
	\label{eq:intro:saa}
	\min_{u\in \adcsp}\,
	\frac{1}{2N}\sum_{i=1}^N \norm[L^2(\domain)]{S(u,\xi^i)}^2
	+(\alpha/2)\norm[L^2(\domain)]{u}^2,
\end{align}
where $\xi^1, \xi^2,\ldots$, are as in \cref{sec:assumptions}.

Let $\hat{\erpobj}_{N}$ be the objective function
of \eqref{eq:intro:saa} and let $C_\domain$ be	
\friedrichs' constant of the domain $\domain$.
For each $(u,\xi) \in L^2(\domain) \times \Xi$, we have
the stability estimate (cf.\ \cite[eqns.\ (2.11)]{Garreis2019a})
\begin{align}
	\label{eq:intro:stateequation}
	\norm[H_0^1(\domain)]{S(u,\xi)} \leq 
	(C_\domain/\kappa_{\min})\norm[L^2(\domain)]{u}.
\end{align}
The optimal value of the  risk-neutral
problem \eqref{eq:intro:ocp} and those
of the corresponding SAA problems \eqref{eq:intro:saa}
are zero, as $S(0,\xi) = 0$ for all $\xi \in \Xi$
and $0 \in \adcsp$.
We define $\varepsilon_{\max} = (C_\domain^2/\kappa_{\min})^2+\alpha$.
Let $\varepsilon > 0$ satisfy 
$\varepsilon \leq \varepsilon_{\max}$.
We define
$
V_{\varepsilon} = 
\big\{\, u \in L^2(\domain) \colon \,
\norm[L^2(\domain)]{u}^2 \leq \tfrac{2\varepsilon}
{(C_\domain^2/\kappa_{\min})^2+\alpha}
\,\big\}
$.
It holds that $V_{\varepsilon} \subset \adcsp$.
For each $u \in V_\varepsilon$, the stability estimate
\eqref{eq:intro:stateequation} 
and \friedrichs' inequality yield
\begin{align*}
	\frac{1}{2N} \sum_{i=1}^N \norm[L^2(\domain)]{S(u,\xi^i)}^2
	+ (\alpha/2)\norm[L^2(\domain)]{u}^2
	&\leq (1/2)(C_\domain^2/\kappa_{\min})^2 \norm[L^2(\domain)]{u}^2
	+ (\alpha/2)\norm[L^2(\domain)]{u}^2
	\\
	&\leq 0 + \varepsilon.
\end{align*}
Hence each $u \in V_\varepsilon$ is an $\varepsilon$-optimal solution
to the SAA problem \eqref{eq:intro:saa}.
The set $V_\varepsilon$ is a closed ball about zero with positive radius
because $\varepsilon > 0$.
Since $L^2(\domain)$ is infinite dimensional, this set is noncompact
\cite[Thm.\ 2.5-5]{Kreyszig1978}.
Therefore, as long as $\tilde{\varepsilon} > 0$, 
the level sets of the SAA objective function,
$\{ \, u \in \adcsp \colon \, \hat{\erpobj}_{N}(u) 
\leq \tilde{\varepsilon} \, \}$, are
noncompact, as they contain the noncompact set $V_\varepsilon$
with 
$\varepsilon = \min\{\tilde{\varepsilon},\varepsilon_{\max}\}$.
In this case, an inf-compactness condition 
(see \cite[p.\ 166]{Shapiro2014}) is violated.
\end{appendix}

\section*{Acknowledgements}
JM thanks Professor Alexander Shapiro for valuable
discussions about the SAA approach.

{\footnotesize\bibliography{saa4pde_consistency}}

\begin{thebibliography}{10}

\bibitem{Adams1975}
{\sc R.~A. Adams}, {\em Sobolev {S}paces}, Academic Press, New York, NY, 1975.

\bibitem{AhmadAli2016}
{\sc A.~Ahmad~Ali, K.~Deckelnick, and M.~Hinze}, {\em Global minima for
  semilinear optimal control problems}, Comput. Optim. Appl., 65 (2016),
  pp.~261--288, \url{https://doi.org/10.1007/s10589-016-9833-1}.

\bibitem{Alexanderian2017}
{\sc A.~Alexanderian, N.~Petra, G.~Stadler, and O.~Ghattas}, {\em Mean-variance
  risk-averse optimal control of systems governed by {PDE}s with random
  parameter fields using quadratic approximations}, SIAM/ASA J. Uncertain.
  Quantif., 5 (2017), pp.~1166--1192, \url{https://doi.org/10.1137/16M106306X}.

\bibitem{Aliprantis2006}
{\sc C.~D. Aliprantis and K.~C. Border}, {\em Infinite {D}imensional
  {A}nalysis: {A} {H}itchhiker's {G}uide}, Springer, Berlin, 3rd~ed., 2006,
  \url{https://doi.org/10.1007/3-540-29587-9}.

\bibitem{Alla2019}
{\sc A.~Alla, M.~Hinze, P.~Kolvenbach, O.~Lass, and S.~Ulbrich}, {\em A
  certified model reduction approach for robust parameter optimization with
  {PDE} constraints}, Adv. Comput. Math., 45 (2019), pp.~1221--1250,
  \url{https://doi.org/10.1007/s10444-018-9653-1}.

\bibitem{Alt2016}
{\sc H.~W. Alt}, {\em Linear {F}unctional {A}nalysis: {A}n
  {A}pplication-{O}riented {I}ntroduction}, Universitext, Springer, London,
  2016, \url{https://doi.org/10.1007/978-1-4471-7280-2}.
\newblock Translated from the German edition by Robert N\"{u}rnberg.

\bibitem{Artstein1995}
{\sc Z.~Artstein and R.~J.~B. Wets}, {\em Consistency of minimizers and the
  {SLLN} for stochastic programs}, J. Convex Anal., 2 (1995), pp.~1--17.

\bibitem{Aubin2009}
{\sc J.-P. Aubin and H.~Frankowska}, {\em Set-{V}alued {A}nalysis}, Mod.
  Birkh\"auser Class., Springer, Boston, 2009,
  \url{https://doi.org/10.1007/978-0-8176-4848-0}.

\bibitem{Banholzer2019}
{\sc D.~Banholzer, J.~Fliege, and R.~Werner}, {\em On rates of convergence for
  sample average approximations in the almost sure sense and in mean}, Math.
  Program., 191 (2022), pp.~307--345,
  \url{https://doi.org/10.1007/s10107-019-01400-4}.

\bibitem{Bauschke2011}
{\sc H.~H. Bauschke and P.~L. Combettes}, {\em Convex {A}nalysis and {M}onotone
  {O}perator {T}heory in {H}ilbert {S}paces}, CMS Books Math., Springer, New
  York, 2011, \url{https://doi.org/10.1007/978-1-4419-9467-7}.

\bibitem{Ben-Tal2009}
{\sc A.~Ben-Tal, L.~El~Ghaoui, and A.~Nemirovski}, {\em Robust {O}ptimization},
  Princeton Ser. Appl. Math., Princeton University Press, Princeton, NJ, 2009.

\bibitem{Billingsley2012}
{\sc P.~Billingsley}, {\em Probability and {M}easure}, Wiley Ser. Probab.
  Stat., John Wiley \& Sons, Hoboken, NJ, 2012.

\bibitem{Bonnans2013}
{\sc J.~F. Bonnans and A.~Shapiro}, {\em Perturbation {A}nalysis of
  {O}ptimization {P}roblems}, Springer Ser. Oper. Res., Springer, New York,
  2000, \url{https://doi.org/10.1007/978-1-4612-1394-9}.

\bibitem{Casas2012c}
{\sc E.~Casas and F.~Tr\"{o}ltzsch}, {\em Second order analysis for optimal
  control problems: {I}mproving results expected from abstract theory}, SIAM J.
  Optim., 22 (2012), pp.~261--279, \url{https://doi.org/10.1137/110840406}.

\bibitem{Castaing1977}
{\sc C.~Castaing and M.~Valadier}, {\em Convex {A}nalysis and {M}easurable
  {M}ultifunctions}, Lecture Notes in Math. 580, Springer, Berlin, 1977,
  \url{https://doi.org/10.1007/bfb0087685}.

\bibitem{Chen2020}
{\sc P.~Chen and O.~Ghattas}, {\em Taylor approximation for chance constrained
  optimization problems governed by partial differential equations with
  high-dimensional random parameters}, SIAM/ASA J. Uncertain. Quantif., 9
  (2021), pp.~1381--1410, \url{https://doi.org/10.1137/20M1381381}.

\bibitem{Conti2011}
{\sc S.~Conti, H.~Held, M.~Pach, M.~Rumpf, and R.~Schultz}, {\em Risk averse
  shape optimization}, SIAM J. Control Optim., 49 (2011), pp.~927--947,
  \url{https://doi.org/10.1137/090754315}.

\bibitem{Cucker2002}
{\sc F.~Cucker and S.~Smale}, {\em On the mathematical foundations of
  learning}, Bull. Amer. Math. Soc. (N.S.), 39 (2002), pp.~1--49,
  \url{https://doi.org/10.1090/S0273-0979-01-00923-5}.

\bibitem{delosReyes2004}
{\sc J.~C. de~los Reyes and K.~Kunisch}, {\em A comparison of algorithms for
  control constrained optimal control of the {B}urgers equation}, Calcolo, 41
  (2004), pp.~203--225, \url{https://doi.org/10.1007/s10092-004-0092-7}.

\bibitem{Farshbaf-Shaker2020}
{\sc M.~H. Farshbaf-Shaker, M.~Gugat, H.~Heitsch, and R.~Henrion}, {\em Optimal
  {N}eumann boundary control of a vibrating string with uncertain initial data
  and probabilistic terminal constraints}, SIAM J. Control Optim., 58 (2020),
  pp.~2288--2311, \url{https://doi.org/10.1137/19M1269944}.

\bibitem{Farshbaf-Shaker2018}
{\sc M.~H. Farshbaf-Shaker, R.~Henrion, and D.~H\"{o}mberg}, {\em Properties of
  chance constraints in infinite dimensions with an application to {PDE}
  constrained optimization}, Set-Valued Var. Anal., 26 (2018), pp.~821--841,
  \url{https://doi.org/10.1007/s11228-017-0452-5}.

\bibitem{Gahururu2019}
{\sc D.~Gahururu, M.~Hinterm\"uller, S.-M. Stengl, and T.~M. Surowiec}, {\em
  Generalized {N}ash equilibrium problems with partial differential operators:
  {T}heory, algorithms, and risk aversion}, in Non-{S}mooth and
  {C}omplementarity-{B}ased {D}istributed {P}arameter {S}ystems: {S}imulation
  and {H}ierarchical {O}ptimization, M.~Hinterm\"uller, R.~Herzog, C.~Kanzow,
  M.~Ulbrich, and S.~Ulbrich, eds., Internat. Ser. Numer. Math. 172,
  Birkh\"auser, Cham, 2022, \url{https://doi.org/10.1007/978-3-030-79393-7_7}.

\bibitem{Garreis2017}
{\sc S.~Garreis and M.~Ulbrich}, {\em Constrained optimization with low-rank
  tensors and applications to parametric problems with {PDE}s}, SIAM J. Sci.
  Comput., 39 (2017), pp.~A25--A54, \url{https://doi.org/10.1137/16M1057607}.

\bibitem{Garreis2019a}
{\sc S.~Garreis and M.~Ulbrich}, {\em A fully adaptive method for the optimal
  control of semilinear elliptic {PDE}s under uncertainty using low-rank
  tensors}, {P}reprint, Technische Universit\"at M\"unchen, M\"unchen, 2019,
  \url{http://go.tum.de/204409}.

\bibitem{Geiersbach2020}
{\sc C.~Geiersbach and T.~Scarinci}, {\em Stochastic proximal gradient methods
  for nonconvex problems in {H}ilbert spaces}, Comput. Optim. Appl., 78 (2021),
  pp.~705--740, \url{https://doi.org/10.1007/s10589-020-00259-y}.

\bibitem{Geletu2020}
{\sc A.~Geletu, A.~Hoffmann, P.~Schmidt, and P.~Li}, {\em Chance constrained
  optimization of elliptic {PDE} systems with a smoothing convex
  approximation}, ESAIM Control Optim. Calc. Var., 26 (2020), pp.~Paper No. 70,
  28, \url{https://doi.org/10.1051/cocv/2019077}.

\bibitem{Guigues2017}
{\sc V.~Guigues, A.~Juditsky, and A.~Nemirovski}, {\em Non-asymptotic
  confidence bounds for the optimal value of a stochastic program}, Optim.
  Methods Softw., 32 (2017), pp.~1033--1058,
  \url{https://doi.org/10.1080/10556788.2017.1350177}.

\bibitem{Guth2019}
{\sc P.~A. Guth, V.~Kaarnioja, F.~Y. Kuo, C.~Schillings, and I.~H. Sloan}, {\em
  A quasi-{M}onte {C}arlo method for optimal control under uncertainty},
  SIAM/ASA J. Uncertain. Quantif., 9 (2021), pp.~354--383,
  \url{https://doi.org/10.1137/19M1294952}.

\bibitem{Haber2012}
{\sc E.~Haber, M.~Chung, and F.~Herrmann}, {\em An effective method for
  parameter estimation with {PDE} constraints with multiple right-hand sides},
  SIAM J. Optim., 22 (2012), pp.~739--757,
  \url{https://doi.org/10.1137/11081126X}.

\bibitem{Hess1996}
{\sc C.~Hess}, {\em Epi-convergence of sequences of normal integrands and
  strong consistency of the maximum likelihood estimator}, Ann. Statist., 24
  (1996), pp.~1298--1315, \url{https://doi.org/10.1214/aos/1032526970}.

\bibitem{Hintermueller2020}
{\sc M.~Hinterm\"uller and S.-M. Stengl}, {\em On the convexity of optimal
  control problems involving non-linear {PDE}s or {VI}s and applications to
  {N}ash games}, 2020, \url{https://doi.org/10.20347/WIAS.PREPRINT.2759}.

\bibitem{Hinze2009}
{\sc M.~Hinze, R.~Pinnau, M.~Ulbrich, and S.~Ulbrich}, {\em Optimization with
  {PDE} {C}onstraints}, Math. Model. Theory Appl. 23, Springer, Dordrecht,
  2009, \url{https://doi.org/10.1007/978-1-4020-8839-1}.

\bibitem{Hoffhues2020}
{\sc M.~Hoffhues, W.~R\"{o}misch, and T.~M. Surowiec}, {\em On quantitative
  stability in infinite-dimensional optimization under uncertainty}, Optim.
  Lett., 15 (2021), pp.~2733--2756,
  \url{https://doi.org/10.1007/s11590-021-01707-2}.

\bibitem{Huber1967}
{\sc P.~J. Huber}, {\em The behavior of maximum likelihood estimates under
  nonstandard conditions}, in Proceedings of the Fifth Berkeley Symposium on
  Mathematical Statistics and Probability, Volume 1: Statistics, L.~M. {Le Cam}
  and J.~Neyman, eds., Berkeley, CA, 1967, University of California Press,
  pp.~221--233.

\bibitem{Hytoenen2016}
{\sc T.~Hyt\"{o}nen, J.~van Neerven, M.~Veraar, and L.~Weis}, {\em Analysis in
  {B}anach {S}paces: {M}artingales and {L}ittlewood-{P}aley {T}heory}, Ergeb.
  Math. Grenzgeb. (3) 63, Springer, Cham, 2016,
  \url{https://doi.org/10.1007/978-3-319-48520-1}.

\bibitem{Kahlbacher2008}
{\sc M.~Kahlbacher and S.~Volkwein}, {\em Estimation of diffusion coefficients
  in a scalar {G}inzburg-{L}andau equation by using model reduction}, in
  Numerical Mathematics and Advanced Applications, K.~Kunisch, G.~Of, and
  O.~Steinbach, eds., Springer, Berlin, 2008, pp.~727--734,
  \url{https://doi.org/10.1007/978-3-540-69777-0_87}.

\bibitem{Kaniovski1995}
{\sc {\relax Yu}.~M. Kaniovski, A.~J. King, and R.~J.-B. Wets}, {\em
  Probabilistic bounds (via large deviations) for the solutions of stochastic
  programming problems}, Ann. Oper. Res., 56 (1995), pp.~189--208,
  \url{https://doi.org/10.1007/BF02031707}.

\bibitem{Kleywegt2002}
{\sc A.~J. Kleywegt, A.~Shapiro, and T.~Homem-de Mello}, {\em The sample
  average approximation method for stochastic discrete optimization}, SIAM J.
  Optim., 12 (2002), pp.~479--502,
  \url{https://doi.org/10.1137/S1052623499363220}.

\bibitem{Kolvenbach2018}
{\sc P.~Kolvenbach, O.~Lass, and S.~Ulbrich}, {\em An approach for robust
  {PDE}-constrained optimization with application to shape optimization of
  electrical engines and of dynamic elastic structures under uncertainty},
  Optim. Eng., 19 (2018), pp.~697--731,
  \url{https://doi.org/10.1007/s11081-018-9388-3}.

\bibitem{Kouri2014a}
{\sc D.~P. Kouri}, {\em A multilevel stochastic collocation algorithm for
  optimization of {PDE}s with uncertain coefficients}, SIAM/ASA J. Uncertain.
  Quantif., 2 (2014), pp.~55--81, \url{https://doi.org/10.1137/130915960}.

\bibitem{Kouri2017}
{\sc D.~P. Kouri}, {\em A measure approximation for distributionally robust
  {PDE}-constrained optimization problems}, SIAM J. Numer. Anal., 55 (2017),
  pp.~3147--3172, \url{https://doi.org/10.1137/15M1036944}.

\bibitem{Kouri2013}
{\sc D.~P. Kouri, M.~Heinkenschloss, D.~Ridzal, and B.~van Bloemen~Waanders},
  {\em A trust-region algorithm with adaptive stochastic collocation for {PDE}
  optimization under uncertainty}, SIAM J. Sci. Comput., 35 (2013),
  pp.~A1847--A1879, \url{https://doi.org/10.1137/120892362}.

\bibitem{Kouri2018}
{\sc D.~P. Kouri and A.~Shapiro}, {\em Optimization of {PDE}s with uncertain
  inputs}, in Frontiers in PDE-Constrained Optimization, H.~Antil, D.~P. Kouri,
  M.-D. Lacasse, and D.~Ridzal, eds., IMA Vol. Math. Appl. 163, Springer, New
  York, NY, 2018, pp.~41--81,
  \url{https://doi.org/10.1007/978-1-4939-8636-1_2}.

\bibitem{Kouri2016}
{\sc D.~P. Kouri and T.~M. Surowiec}, {\em Risk-averse {PDE}-constrained
  optimization using the conditional value-at-risk}, SIAM J. Optim., 26 (2016),
  pp.~365--396, \url{https://doi.org/10.1137/140954556}.

\bibitem{Kouri2018a}
{\sc D.~P. Kouri and T.~M. Surowiec}, {\em Existence and optimality conditions
  for risk-averse {PDE}-constrained optimization}, SIAM/ASA J. Uncertain.
  Quantif., 6 (2018), pp.~787--815, \url{https://doi.org/10.1137/16M1086613}.

\bibitem{Kouri2019a}
{\sc D.~P. Kouri and T.~M. Surowiec}, {\em Epi-regularization of risk
  measures}, Math. Oper. Res., 45 (2020), pp.~774--795,
  \url{https://doi.org/10.1287/moor.2019.1013}.

\bibitem{Kouri2020}
{\sc D.~P. Kouri and T.~M. Surowiec}, {\em Risk-averse optimal control of
  semilinear elliptic {PDE}s}, ESAIM Control. Optim. Calc. Var., 26 (2020),
  \url{https://doi.org/10.1051/cocv/2019061}.

\bibitem{Kouri2020a}
{\sc D.~P. Kouri and T.~M. Surowiec}, {\em A primal--dual algorithm for risk
  minimization}, Math. Program., 193 (2022), pp.~337--363,
  \url{https://doi.org/10.1007/s10107-020-01608-9}.

\bibitem{Kreyszig1978}
{\sc E.~Kreyszig}, {\em Introductory {F}unctional {A}nalysis with
  {A}pplications}, John Wiley \& Sons, New York, NY, 1978.

\bibitem{Lachout2005}
{\sc P.~Lachout, E.~Liebscher, and S.~Vogel}, {\em Strong convergence of
  estimators as {$\epsilon_n$}-minimisers of optimisation problems}, Ann. Inst.
  Statist. Math., 57 (2005), pp.~291--313,
  \url{https://doi.org/10.1007/BF02507027}.

\bibitem{Lass2017}
{\sc O.~Lass and S.~Ulbrich}, {\em Model order reduction techniques with a
  posteriori error control for nonlinear robust optimization governed by
  partial differential equations}, SIAM J. Sci. Comput., 39 (2017),
  pp.~S112--S139, \url{https://doi.org/10.1137/16M108269X}.

\bibitem{LeCam1953}
{\sc L.~M. {Le Cam}}, {\em On some asymptotic properties of maximum likelihood
  estimates and related {B}ayes' estimates}, Univ. California Publ. Stat. 1,
  (1953), pp.~277--329, \url{https://hdl.handle.net/2027/wu.89045844305}.

\bibitem{Mannel2020}
{\sc F.~Mannel and A.~Rund}, {\em A hybrid semismooth quasi-{N}ewton method for
  nonsmooth optimal control with {PDE}s}, Optim. Eng., 22 (2021),
  pp.~2087--2125, \url{https://doi.org/10.1007/s11081-020-09523-w}.

\bibitem{Martin2021}
{\sc M.~Martin, S.~Krumscheid, and F.~Nobile}, {\em Complexity analysis of
  stochastic gradient methods for {PDE}-constrained optimal control problems
  with uncertain parameters}, ESAIM Math. Model. Numer. Anal., 55 (2021),
  pp.~1599--1633, \url{https://doi.org/10.1051/m2an/2021025}.

\bibitem{MartnezFrutos2018}
{\sc J.~Mart{\'{\i}}nez-Frutos and F.~P. Esparza}, {\em Optimal {C}ontrol of
  {PDEs} under {U}ncertainty: {A}n {I}ntroduction with {A}pplication to
  {O}ptimal {S}hape {D}esign of {S}tructures}, SpringerBriefs Math., Springer,
  Cham, 2018, \url{https://doi.org/10.1007/978-3-319-98210-6}.

\bibitem{Meidner2008}
{\sc D.~Meidner and B.~Vexler}, {\em A priori error estimates for space-time
  finite element discretization of parabolic optimal control problems. {P}art
  \textnormal{II}: {P}roblems with control constraints}, SIAM J. Control
  Optim., 47 (2008), pp.~1301--1329, \url{https://doi.org/10.1137/070694028}.

\bibitem{Milz2021a}
{\sc J.~Milz}, {\em Topics in {PDE}-{C}onstrained {O}ptimization under
  {U}ncertainty and {U}ncertainty {Q}uantification}, {D}issertation, Technische
  Universit\"at M\"unchen, M\"unchen, 2021.

\bibitem{Milz2021}
{\sc J.~Milz}, {\em Sample average approximations of strongly convex stochastic
  programs in {H}ilbert spaces}, Optim. Lett.,  (2022),
  \url{https://doi.org/10.1007/s11590-022-01888-4}.

\bibitem{Milz2020a}
{\sc J.~Milz and M.~Ulbrich}, {\em An approximation scheme for distributionally
  robust {PDE}-constrained optimization}, SIAM J. Control Optim., 60 (2022),
  pp.~1410--1435, \url{https://doi.org/10.1137/20M134664X}.

\bibitem{Nasir2021}
{\sc Y.~Nasir, O.~Volkov, and L.~J. Durlofsky}, {\em A two-stage optimization
  strategy for large-scale oil field development}, Optim. Eng.,  (2021),
  \url{https://doi.org/10.1007/s11081-020-09591-y}.

\bibitem{Necas2012}
{\sc J.~Ne\v{c}as}, {\em Direct {M}ethods in the {T}heory of {E}lliptic
  {E}quations}, Springer Monogr. Math., Springer, Heidelberg, 2012,
  \url{https://doi.org/10.1007/978-3-642-10455-8}.

\bibitem{Phelps2016}
{\sc C.~Phelps, J.~Royset, and Q.~Gong}, {\em Optimal control of uncertain
  systems using sample average approximations}, SIAM J. Control Optim., 54
  (2016), pp.~1--29, \url{https://doi.org/10.1137/140983161}.

\bibitem{Pieper2015}
{\sc K.~Pieper}, {\em Finite element discretization and efficient numerical
  solution of elliptic and parabolic sparse control problems}, {D}issertation,
  Technische Universit\"at M\"unchen, M\"unchen, 2015,
  \url{http://nbn-resolving.de/urn/resolver.pl?urn:nbn:de:bvb:91-diss-20150420-1241413-1-4}.

\bibitem{Polak1997}
{\sc E.~Polak}, {\em Optimization: {A}lgorithms and {C}onsistent
  {A}pproximations}, Appl. Math. Sci. 124, Springer, New York, 1997,
  \url{https://doi.org/10.1007/978-1-4612-0663-7}.

\bibitem{Renardy2004}
{\sc M.~Renardy and R.~C. Rogers}, {\em An {I}ntroduction to {P}artial
  {D}ifferential {E}quations}, Texts Appl. Math. 13, Springer, New York, NY,
  2nd~ed., 2004, \url{https://doi.org/10.1007/b97427}.

\bibitem{Roemisch2021}
{\sc W.~R\"omisch and T.~M. Surowiec}, {\em Asymptotic properties of {M}onte
  {C}arlo methods in elliptic {PDE}-constrained optimization under
  uncertainty}, preprint, \url{https://arxiv.org/abs/2106.06347}, 2021.

\bibitem{Royset2019}
{\sc J.~O. Royset}, {\em Approximations of semicontinuous functions with
  applications to stochastic optimization and statistical estimation}, Math.
  Program., 184 (2020), pp.~289--318,
  \url{https://doi.org/10.1007/s10107-019-01413-z}.

\bibitem{Shapiro1991}
{\sc A.~Shapiro}, {\em Asymptotic analysis of stochastic programs}, Ann. Oper.
  Res., 30 (1991), pp.~169--186, \url{https://doi.org/10.1007/BF02204815}.

\bibitem{Shapiro1993}
{\sc A.~Shapiro}, {\em Asymptotic behavior of optimal solutions in stochastic
  programming}, Math. Oper. Res., 18 (1993), pp.~829--845,
  \url{https://doi.org/10.1287/moor.18.4.829}.

\bibitem{Shapiro2003}
{\sc A.~Shapiro}, {\em Monte {C}arlo {S}ampling {M}ethods}, in Stochastic
  {P}rogramming, Handbooks in Oper. Res. Manag. Sci. 10, Elsevier, 2003,
  pp.~353--425, \url{https://doi.org/10.1016/S0927-0507(03)10006-0}.

\bibitem{Shapiro2008}
{\sc A.~Shapiro}, {\em Stochastic programming approach to optimization under
  uncertainty}, Math. Program., 112 (2008), pp.~183--220,
  \url{https://doi.org/10.1007/s10107-006-0090-4}.

\bibitem{Shapiro2014}
{\sc A.~Shapiro, D.~Dentcheva, and A.~Ruszczy{\'{n}}ski}, {\em Lectures on
  {S}tochastic {P}rogramming: {M}odeling and {T}heory}, MOS-SIAM Ser. Optim.,
  SIAM, Philadelphia, PA, 2nd~ed., 2014,
  \url{https://doi.org/10.1137/1.9781611973433}.

\bibitem{Shapiro2005}
{\sc A.~Shapiro and A.~Nemirovski}, {\em On complexity of stochastic
  programming problems}, in Continuous Optimization: Current Trends and Modern
  Applications, V.~Jeyakumar and A.~Rubinov, eds., Appl. Optim. 99, Springer,
  Boston, MA, 2005, pp.~111--146,
  \url{https://doi.org/10.1007/0-387-26771-9_4}.

\bibitem{Tiesler2012}
{\sc H.~Tiesler, R.~M. Kirby, D.~Xiu, and T.~Preusser}, {\em Stochastic
  collocation for optimal control problems with stochastic {PDE} constraints},
  SIAM J. Control Optim., 50 (2012), pp.~2659--2682,
  \url{https://doi.org/10.1137/110835438}.

\bibitem{Tong2022}
{\sc S.~Tong, A.~Subramanyam, and V.~Rao}, {\em Optimization under rare chance
  constraints}, SIAM J. Optim., 32 (2022), pp.~930--958,
  \url{https://doi.org/10.1137/20M1382490}.

\bibitem{Troeltzsch2010a}
{\sc F.~Tr\"oltzsch}, {\em Optimal {C}ontrol of {P}artial {D}ifferential
  {E}quations: {T}heory, Methods and Applications}, Grad. Stud. Math. 112, AMS,
  Providence, RI, 2010, \url{https://doi.org/10.1090/gsm/112}.
\newblock Translated by J. Sprekels.

\bibitem{Ulbrich2011}
{\sc M.~Ulbrich}, {\em Semismooth {N}ewton {M}ethods for {V}ariational
  {I}nequalities and {C}onstrained {O}ptimization {P}roblems in {F}unction
  {S}paces}, MOS-SIAM Ser. Optim., SIAM, Philadelphia, PA, 2011,
  \url{https://doi.org/10.1137/1.9781611970692}.

\bibitem{Volkwein1997}
{\sc S.~Volkwein}, {\em {M}esh-{I}ndependence of an {A}ugmented
  {L}agrangian-{SQP} {M}ethod in {H}ilbert {S}paces and {C}ontrol {P}roblems
  for the {B}urgers {E}quation}, {D}issertation, Technical University of
  Berlin, Berlin, 1997, \url{https://imsc.uni-graz.at/volkwein/diss.ps}.

\bibitem{Volkwein2000a}
{\sc S.~Volkwein}, {\em Application of the augmented {L}agrangian-{SQP} method
  to optimal control problems for the stationary {B}urgers equation}, Comput.
  Optim. Appl., 16 (2000), pp.~57--81,
  \url{https://doi.org/10.1023/A:1008777520259}.

\bibitem{Volkwein2000}
{\sc S.~Volkwein}, {\em Mesh-independence for an augmented {L}agrangian-{SQP}
  method in {H}ilbert spaces}, SIAM J. Control Optim., 38 (2000), pp.~767--785,
  \url{https://doi.org/10.1137/S0363012998334468}.

\bibitem{Wechsung2021}
{\sc F.~Wechsung, A.~Giuliani, M.~Landreman, A.~J. Cerfon, and G.~Stadler},
  {\em Single-stage gradient-based stellarator coil design: stochastic
  optimization}, Nuclear Fusion, 62 (2022), p.~076034,
  \url{https://doi.org/10.1088/1741-4326/ac45f3}.

\bibitem{Yang2017}
{\sc H.~Yang and M.~Gunzburger}, {\em Algorithms and analyses for stochastic
  optimization for turbofan noise reduction using parallel reduced-order
  modeling}, Comput. Methods Appl. Mech. Engrg., 319 (2017), pp.~217--239,
  \url{https://doi.org/10.1016/j.cma.2017.02.030}.

\end{thebibliography}

\end{document}